\documentclass[final,onefignum,onetabnum]{siamart171218}

\usepackage{amsmath, amssymb}
\usepackage{lipsum}
\usepackage{amsfonts}
\usepackage{booktabs}
\usepackage{comment}
\usepackage{mathtools}
\usepackage{subfigure}

\usepackage{mathrsfs}
\usepackage{graphicx}
\usepackage{epstopdf}
\usepackage{algorithmic}

\usepackage[normalem]{ulem} 

\numberwithin{equation}{section}
\numberwithin{table}{section}
\numberwithin{figure}{section}

\ifpdf
  \DeclareGraphicsExtensions{.eps,.pdf,.png,.jpg}
\else
  \DeclareGraphicsExtensions{.eps}
\fi

\newsiamremark{remark}{Remark}
\newsiamremark{hypothesis}{Hypothesis}
\newsiamremark{assumption}{Assumption}
\crefname{hypothesis}{Hypothesis}{Hypotheses}
\newsiamthm{claim}{Claim}

\newcommand\normt[1]{\lVert#1\rVert_2}
\newcommand\normF[1]{\lVert#1\rVert_F}

\newcommand\norm[1]{\lVert#1\rVert}

\DeclareMathOperator{\rank}{rank}

\def\new#1{\textcolor{black}{#1}}

\usepackage{scalerel}

\let\bigopsize\bigoplus
\def\bigoplus{{\scalerel*{\boldsymbol\oplus}{\bigopsize}}}

\headers{Mixed Precision HODLR Matrices}{Erin Carson, Xinye Chen, and Xiaobo Liu}

\title{Mixed Precision HODLR Matrices\thanks{Version of December 5, 2024.
\funding{The first and second author acknowledge funding from the European Union (ERC, inEXASCALE, 101075632). Views and opinions expressed are those of the authors only and do not necessarily reflect those of the European Union or the European Research Council. Neither the European Union nor the granting authority can be held responsible for them.  The first author additionally acknowledges funding from the Charles University Research Centre program No. UNCE/24/SCI/005. 
The second author acknowledges funding from the France 2030 NumPEx Exa-MA (ANR-22-EXNU-0002) project managed by the French National Research Agency (ANR). }
}}

\author{Erin Carson\thanks{Department of Numerical Mathematics, Charles University, Prague, Czech Republic
  (\email{carson@karlin.mff.cuni.cz}), 
  }
\and Xinye Chen\thanks{LIP6, Sorbonne Université, CNRS, Paris, France
  (\email{xinye.chen@lip6.fr}), 
  }
\and Xiaobo Liu\thanks{Max Planck Institute for Dynamics of Complex Technical Systems, Magdeburg, Germany (\email{xliu@mpi-magdeburg.mpg.de}).}}

\usepackage{amsopn}

\makeatletter
\newcommand*{\addFileDependency}[1]{
  \typeout{(#1)}
  \@addtofilelist{#1}
  \IfFileExists{#1}{}{\typeout{No file #1.}}
}
\makeatother

\newcommand*{\myexternaldocument}[1]{%
    \externaldocument{#1}%
    \addFileDependency{#1.tex}%
    \addFileDependency{#1.aux}%
}

\ifpdf
\hypersetup{
  pdftitle={Mixed precision HODLR matrices},
  pdfauthor={E. Carson, X. Chen, and X. Liu}
}
\fi

\myexternaldocument{ex_supplement}

\begin{document}

\maketitle

\begin{abstract}
Hierarchical matrix computations have attracted significant attention in the science and engineering community as exploiting  data-sparse structures can significantly reduce the computational complexity of many important kernels. One particularly popular option within this class is the Hierarchical Off-Diagonal Low-Rank (HODLR) format. 
In this paper, we show that the off-diagonal blocks of HODLR matrices that are approximated by low-rank matrices can be represented in low precision without degenerating the quality of the overall approximation (with the error growth bounded by a factor of $2$). We also present an adaptive-precision scheme for constructing and storing HODLR matrices, and we prove that the use of mixed precision does not compromise the numerical stability of the resulting HODLR matrix--vector product and LU factorization. That is, the resulting error in these computations is not significantly greater than the case where we use one precision (say, double) for constructing and storing the HODLR matrix. 
Our analyses further give insight on how one must choose the working precision in HODLR matrix computations relative to the approximation error in order to not observe the effects of finite precision. Intuitively, when a HODLR matrix is subject to a high degree of approximation error, subsequent computations can be performed in a lower precision without detriment. We demonstrate the validity of our theoretical results through a range of numerical experiments.  
\end{abstract}



\section{Introduction}
\label{sec:intro}
The costs of standard linear algebraic operations (matrix--vector products, matrix factorizations, matrix inversion, etc.) grow prohibitively with matrix size. Accelerating such operations typically requires exploiting special structure of the matrix.  Finite element discretization of differential equations, e.g., time-space diffusion equations \cite{begr97, beha03, kbr11, doi:10.1137/18M1180803} and boundary integral equations \cite{HACKBUSCH2002129, doi:10.1137/1.9781611976045}, often result in rank-structured matrices, most of whose off-diagonal blocks are often of low rank. Rank-structured matrices are often represented in Block Low Rank (BLR) format \cite{doi:10.1137/16M1077192} or hierarchical formats (e.g., $\mathcal{H}$-matrices \cite{beha03, BORM2003405}, $\mathcal{H}^2$-matrices \cite{BORM2003405, 10.1007/978-3-642-59709-1_2}, HSS matrices \cite{amda13, doi:10.1137/S0895479803436652}), whose operations often reduce the complexity of matrix computations.  

Hierarchical matrices \cite{10.5555/1481537}, often abbreviated as $\mathcal{H}$-matrices, comprise a class of dense rank-structured matrices with a hierarchical low-rank structure, which is used to approximate a dense matrix by dividing it into multiple submatrices in a hierarchical way, where a number of submatrices are selected to be approximated by low-rank factors according to an admissibility condition. Low-rank approximation techniques involving the singular value decomposition \cite[sect.~2.4]{GoluVanl96}, QR factorization \cite[sect.~5.2]{GoluVanl96}, interpolative decomposition as well as their randomized variants \cite{doi:10.1137/100786617}, and adaptive cross approximation \cite{Rjasanow}, have been widely studied. $\mathcal{H}$-matrices have a wide range of applications including but not limited to, linear system solvers \cite{amda13, 10046094, doi:10.1137/18M1180803}, radial basis function interpolation \cite{amda13, doi:10.1137/22M1491253},  matrix functions~\cite{casulli2024computing, 23370}, matrix equations~\cite{babe06, ghk03, doi:10.1137/17M1161038}, matrix factorizations~\cite{doi:10.1137/16M1077192, Benner2010, kressner2018fast, Lintner2004}, preconditioning \cite{doi:10.1137/20M1365958}, kernel approximation \cite{CHEN2019108828, doi:10.1137/16M1101167, 8425507}, and $n$-body problems and particle simulations \cite{Barnes1986, GREENGARD1987325}. Employing $\mathcal{H}$-matrices often results in reduced runtime and enables the storage of a variety of dense matrices with linear-polylogarithmic complexity, e.g.,  a square matrix--vector product is reduced from $O(n^2)$ to $O(n\log n)$ in both time and space complexity.

\begin{table}[h]
\caption{Units roundoff of various floating point formats; $u$ denotes the unit roundoff corresponding to the precision, $x_{\min}$ denotes the smallest positive normalized floating-point number, $x_{\max}$ denotes the largest floating-point number, $t$ denotes the number of binary digits in the significand (including the implicit leading bit), $e_{\min}$ denotes exponent of $x_{\min}$, and $e_{\max}$ denotes exponent of $x_{\max}$.} 
\label{table:unitroundoff} 
\centering \setlength\tabcolsep{5pt}\scriptsize
\begin{tabular}{l l l l r r r} 
\hline\\[-1.5mm]
& \qquad $u$      & \quad $x_{\min}$  &    \quad $x_{\max}$  &  $t$    &  $e_{\min}$  & $e_{\max}$\\[1mm] \hline\\[-0.6mm]
quater precision (q52) &  $1.25 \times 10^{-1}$  &  $6.10 \times 10^{-5}$&   $5.73 \times 10^{4}$ &        -16  &    -14   &   15 \\
 bfloat16 (bf16) &   $3.91\times 10^{-3}$  &  $1.18\times 10^{-38}$  & $3.39\times 10^{38}$   & 8     &   -126 & 127 \\
 half precision (fp16)  &  $4.88\times 10^{-4}$ &   $6.10 \times 10^{-5}$ &  $6.55\times 10^{4}$&   11   &   -14 &     15 \\
 single (fp32)  &  $5.96\times 10^{-8}$  &  $1.18\times 10^{-38}$ &   $3.40\times 10^{38}$ &  24    &  -126 &    127 \\
 double (fp64)  &  $1.11\times 10^{-16}$  &  $2.23\times 10^{-308}$ & $1.80\times 10^{308}$ &  53  &   -1022  &  1023 \\[1mm]
\hline 
\end{tabular}
\end{table}

Technically, a matrix with a low-rank off-diagonal structure can be represented in a hierarchical matrix format. One commonly-used format is Hierarchical Off-Diagonal Low-Rank (HODLR), which is a type of $\mathcal{H}$-matrix with weak admissibility conditions. HODLR matrices are formulated by a fixed hierarchical block structure in terms of a binary cluster tree, associated with low-rank approximation of all off-diagonal blocks in each level of the tree which are not further refined. Employing the HODLR matrix format often results in faster matrix computations compared to traditional arithmetic, e.g., the computational complexity of the HODLR matrix--vector product takes $O(pn \log n)$, where $p$ is the maximum rank of the low-rank approximations.  
\new{This reveals the significant role that the maximum rank $p$ plays in the computational complexity of operations on a HODLR matrix. 
For example, the ranks of the HODLR representation of the dense matrices from $n$-body problems in 1D scale like $O(\log n)$, enabling 
matrix--vector products in HODLR representation to be performed with
a complexity of $O(n\log^2 n)$, i.e., almost linear in the data size. 
However, the HODLR ranks for this class of matrices will scale at a faster rate with the data size for problems in higher dimensions, $O(\sqrt{n})$ in 2D and $O(n^{3/2})$ in 3D~\cite{A2024}, and 
as such, the benefits of using the HODLR matrix
representation for computations are diminished.
There are many techniques to improve the performance of the HODLR format in specific domains,} for example,  HODLR with random sampling \cite{doi:10.1137/15M1016679}, HODLR2D \cite{doi:10.1137/22M1491253}, and HODLR3D \cite{A2024} for $n$-body problems. Some high performance software has also been developed; see, e.g., \cite{Ambikasaran2019, mrk20}.

It is known that using low precision arithmetic, now widely available in commercial hardware, often leads to increased performance, a reduced amount of data stored and transferred, as well as reduced energy consumption. 
The use of low precision has emerged as a key strategy to combat data-movement bottlenecks and improve performance in high-performance architectures \cite{10.1145/2503210.2503283}. 
On AMD’s Radeon VII GPU, for example, the compute performance in single precision is about 4$\times$ that in double precision \cite{doi:10.1177/10943420211003313}. Commonly-used floating point formats are summarized in \tablename~\ref{table:unitroundoff}.

It is important to note that the use of low precision can potentially lead to a loss of accuracy and suffers from a reduced range of representable numbers. The key is thus to devise \emph{mixed precision} approaches, in which different precisions are selectively used in different parts of a computation in order to improve performance while maintaining guarantees on accuracy and stability. 
There has thus been much recent work in utilizing mixed precision arithmetic within numerical linear algebra algorithms; see \cite{doi:10.1177/10943420211003313, Higham_Mary_2022} for a survey. For example, \cite{10.1007/978-3-319-93698-7_45} shows that an efficient use of half precision arithmetic Tensor Cores deployed in NVIDIA V100 PCIe GPUs, leads to up to 4$\times$ speedups, with about 80\% reduction in the energy usage compared to highly-optimized linear system solvers. These benefits as well as the increasing support for half-precision arithmetic in modern hardware (e.g., NVIDIA P100 GPU,  Google's Tensor Processing Units, the ARM NEON architecture, and Fujitsu A64FX ARM processor) make lower precision arithmetic appealing in many scientific and engineering applications. 

There are a few existing works on mixed precision implementations and performance analysis of computations with hierarchical matrices. 
For example, \cite{kriemann2024performance} has demonstrated the benefits of $\mathcal{H}$-matrix--vector multiplication arithmetic using low precision compression. 
The work~\cite{bkmt10} introduces the hierarchical diagonal blocking (HDB) representation for sparse matrices and shows that reduced precision can be used to accelerate sparse matrix--vector multiplications (SpMV) under such a blocking scheme, and~\cite{10.1145/3368474.3368479} studies the effect of mixed precision $\mathcal{H}$-matrix--vector multiplications in the boundary element method using the binary64 (fp64) and binary32 (fp32) arithmetic operations.  
Mixed precision computation has also been exploited in the tile low-rank (TLR) matrix format for dense matrix--vector multiplications 
associated with multidimensional convolution operators~\cite{rhlk22} and for Cholesky factorization during the log-likelihood function evaluation in environmental applications~\cite{alsg19}.
However, these prior works are largely application- and implementation-focused, and are missing a rounding error analysis of hierarchical matrix construction and computations in mixed precision arithmetic. 

In this paper, we consider constructing HODLR matrices in a mixed precision manner and offer insights into the resulting behavior of finite precision computations. Our analysis confirms what is largely intuitive: the lower the quality of the low-rank approximation, the lower the precision which can be used without detriment. We provide theoretical bounds which determine which precisions can safely be used in order to balance the overall error. Our primary contributions are as follows:
\begin{itemize}
    \item We develop a mixed precision algorithm for constructing HODLR matrices, and we analyze the global approximation error. Our analysis shows that as the tree \new{level} increases, the unit roundoff must be smaller to offset the error between the HODLR matrix and the original matrix (see Theorem~\ref{thm:mpblr-uv}). 
    \item Based on our error analysis, we propose an adaptive scheme for precision selection, which dynamically determines what degree of precision is required for the computations in each level of the cluster tree. 
    \item We analyze the backward error in computing matrix--vector products 
    and LU factorizations with mixed precision HODLR matrices constructed using our adaptive scheme. Our analyses show how the working precision $u$ should be chosen relative to the approximation parameter $\varepsilon$, giving insight into how these errors should be balanced. Our finite precision analysis remains valid in the case where the HODLR matrices are stored in one precision and therefore also provides new results for this case.
    \item We perform a series of simulations across various datasets to verify our theoretical results. Our code for constructing mixed precision HODLR matrices and fully reproducible experimental code is publicly available at \url{https://github.com/inEXASCALE/mphodlr_exp}.
\end{itemize}

The rest of the paper is organized as follows; 
Section~\ref{sec:hodlr} describes the HODLR matrix representation and the approximation error in HODLR matrices. Section~\ref{sec:mixhodlr} presents our adaptive-precision algorithm for constructing and storing HODLR matrices and analyzes the resulting global representation error. In this section we also give backward error analyses for matrix--vector products and LU factorization with mixed precision HODLR matrices, which come with constraints on the working precision relative to the approximation parameter. To verify our theoretical results, simulations are performed in Section~\ref{sec:nexp}, and the paper is concluded in Section~\ref{sec:conclude}.

We will use the standard model of floating-point arithmetic in \cite[sect.~2.2]{high:ASNA2}. As is standard in the literature, we use the phrase ``precision $u$'' to mean ``precision with unit roundoff $u$''. Given an integer $n$, we define 
$\gamma_n := nu/(1 - nu)$.

\new{Unless otherwise specified, we use hats to denote values computed in floating-point arithmetic, e.g., $\widehat{X}$ denotes $X$ computed in finite precision. 
We also use $\mathrm{fl}(\cdot)$ with an argument that is an arithmetic expression to denote the computed quantity of that expression.
We use the notation $H\equiv A$ to denote the matrix defined in a hierarchical way while using $A$ to emphasize the matrix without explicit hierarchical structure.}
We use the notation $\lesssim$ when 
dropping negligible terms of second order or higher 
in the error bounds; in particular, the threshold $\varepsilon>u$
is viewed as having the same order as the working precision $u$ throughout: for example,
the term $\varepsilon u$ is not necessarily negligible in the expression $\varepsilon^2 +\varepsilon u$, but can be safely dropped from the expression $u +\varepsilon u$.
We will denote by $\|\cdot\|$ any consistent matrix norm, though we will mainly use the Frobenius norm, $\|A\|_F=(\sum_{i,j} |a_{ij}|^2)^{1/2}$.

\section{HODLR Matrices}\label{sec:hodlr}

A matrix $H \in \mathbb{R}^{n \times n}$ is a HODLR matrix if the off-diagonal blocks are low rank and
the diagonal blocks have a similar off-diagonal low-rank structure, which can often further be represented as a HODLR matrix or have sufficiently small sizes.
The computation of HODLR matrix is often performed via a recursive block partition
associated with a binary tree, whose definition is given in Definition~\ref{def.binary-tree}.

\begin{definition}[{\cite[Def.~2.1]{mrk20}}]\label{def.binary-tree}
    A completely balanced binary tree $\mathcal{T}_{\ell}$ of depth $\ell$, whose nodes are subsets of $\{1, \ldots, n\}$, is a cluster tree if
    \begin{enumerate}
        \item[(a)] its root is $I_1^{0}:=I=\{1, \ldots, n\}$;
        \item[(b)] the nodes at level $k$, denoted by $I_1^{k}, \ldots, I_{2^k}^{k}$, form a partitioning of $\{1, \ldots, n\}$ into consecutive indices:
        $I_i^k = \{n_{i-1}^{k}+1, \ldots, n_{i}^{k}-1, n_{i}^{k}\}$
        for some integers $0=n_0^{k} \le n_1^{k}\le\cdots \le n_{2^k}^{k}=n,\ k = 0, \ldots, \ell$. In particular, if $n_{i-1}^{k} = n_i^{k}$ then $I_{i}^{k}=\emptyset$; and
        \item[(c)] the node $I_i^k$ has children $I_{2i-1}^{k+1}$
        and $I_{2i}^{k+1}$ for any $1\le k \le \ell -1$. The
        children are partitioned from their parent.
    \end{enumerate}
\end{definition}

Often, the cluster tree $\mathcal{T}_{\ell}$ is represented in a balanced manner in the sense that the cardinalities of the index sets on the same level are nearly equal and the depth of the tree is determined by a minimal diagonal block size $n_{\min}\ge 1$ for stopping the recursion. 
\new{To simplify the notation and presentation, throughout our analysis it is assumed that a perfectly balanced binary tree is used}, which means $n = 2^{\ell}n_{\min}$.
For an $\ell$-level HODLR format of the matrix $A \in \mathbb{R}^{n \times n}$, the $i$th diagonal block at level $k$, where $1 \le i \le 2^k$ and $0 \le k < \ell$, 
is given by~\cite{amda13}
 \begin{equation}\label{eq:hodlr_block}
     H_{i, i}^{(k)} = \begin{bmatrix}
         H_{2i-1, 2i-1}^{(k+1)} & H_{2i-1, 2i}^{(k+1)}\\
         H_{2i, 2i-1}^{(k+1)} & H_{2i, 2i}^{(k+1)}\\
     \end{bmatrix} \approx 
     \begin{bmatrix}
         H_{2i-1, 2i-1}^{(k+1)} & \widetilde{U}_{2i-1}^{(k+1)} (\widetilde{V}_{ 2i}^{(k+1)})^T\\
         \widetilde{U}_{2i}^{(k+1)} (\widetilde{V}_{2i-1}^{(k+1)})^T & H_{2i, 2i}^{(k+1)}\\
     \end{bmatrix}
     =\widetilde{H}_{i, i}^{(k)},
 \end{equation}
where $H_{i, i}^{(k)} \in \mathbb{R}^{n/2^k \times n / 2^k}$, $\widetilde{U}_{2i-1}^{(k)}, \widetilde{U}_{2i}^{(k)}, \widetilde{V}_{2i-1}^{(k)}, \widetilde{V}_{2i}^{(k)} \in \mathbb{R}^{n / 2^k \times p}$ and $p \ll n$. 
In particular, $H_{1, 1}^{(0)} = A$. The $H_{2i-1, 2i-1}^{(k+1)}$ and $H_{2i, 2i}^{(k+1)}$ blocks can be further treated as HODLR matrices. 
See Fig~\ref{fig:normdist} for examples of HODLR matrix of depth
$\ell=6$ (represented in mixed precision).

Associated with the cluster tree $\mathcal{T}_{\ell}$, one can define a $(\mathcal{T}_{\ell}, p)$-HODLR matrix as follows. 

\begin{definition}[$(\mathcal{T}_{\ell}, p)$-HODLR matrix {\cite[Def.~2.2]{mrk20}}]\label{def:hodlr}
Let the matrix $H \in \mathbb{R}^{n \times n}$ be a $(\mathcal{T}_{\ell}, p)$-HODLR matrix. Then every off-diagonal block $H(I_i^{k},I_j^{k})$ associated with siblings $I_i^{k}$ and $I_j^{k}$ in $\mathcal{T}_{\ell}$, $k = 1, \ldots, \ell$, has rank at most $p$. The HODLR rank of $H$ generated by $\mathcal{T}_{\ell}$ is the smallest integer $p$ such that $H$ is a $(\mathcal{T}_{\ell}, p)$-HODLR matrix.
\end{definition}

As the level of the tree increases, the off-diagonal blocks of a $(\mathcal{T}_{\ell}, p)$-HODLR matrix
in the partition have smaller dimension, and this means 
the rank constraint $p$ becomes less restrictive for them.
We can extend the definition of a $(\mathcal{T}_{\ell}, p)$-HODLR matrix in Definition~\ref{def:hodlr} to a more practical setting, in the sense that 
each off-diagonal block of a matrix in the new class is 
close to the corresponding block of the associated $(\mathcal{T}_{\ell}, p)$-HODLR matrix (which has rank at most $p$) with the relative difference $0\le\varepsilon<1$.
This motivates the definition of a 
$(\mathcal{T}_{\ell}, p, \varepsilon)$-HODLR matrix, given as Definition~\ref{def:hodlr-eps}.

\begin{definition}[$(\mathcal{T}_{\ell}, p, \varepsilon)$-HODLR matrix]\label{def:hodlr-eps}
Let $H \in \mathbb{R}^{n \times n}$ be a $(\mathcal{T}_{\ell}, p)$-HODLR matrix. Then $\widetilde{H} \in \mathbb{R}^{n \times n}$ is defined to be a $(\mathcal{T}_{\ell}, p, \varepsilon)$-HODLR matrix  to $H$, if every off-diagonal block $\widetilde{H}(I_i^{k},I_j^{k})$ associated with siblings $I_i^{k}$ and $I_j^{k}$ in $\mathcal{T}_{\ell}$, $k = 1, \ldots, \ell$, 
satisfies $\|\widetilde{H}(I_i^{k},I_j^{k}) - H(I_i^{k},I_j^{k})\| \le \varepsilon \|H(I_i^{k},I_j^{k})\|$, where $0\le\varepsilon<1$.
\end{definition}

It is clear from Definition~\ref{def:hodlr-eps} that every $(\mathcal{T}_{\ell}, p, \varepsilon)$-HODLR matrix is associated with a $(\mathcal{T}_{\ell}, p)$-HODLR matrix and the two matrices are identical when $\varepsilon=0$. 
However, for a given 
$(\mathcal{T}_{\ell}, p)$-HODLR matrix $H$, the 
associated $(\mathcal{T}_{\ell}, p, \varepsilon)$-HODLR matrix $\widetilde{H}$ can allow, in some levels, off-diagonal blocks of rank exceeding $p$, and, in general,
not much can be said about the rank constraint for these 
off-diagonal blocks except that it now depends somehow on the choice of $\varepsilon$.
It will become clear later that our usage of Definition~\ref{def:hodlr-eps} is for quantifying the error incurred in the low-rank factorization of the 
off-diagonal blocks $\widetilde{H}(I_i^{k},I_j^{k})$ associated with siblings $I_i^{k}$ and $I_j^{k}$ in $\mathcal{T}_{\ell}$, $k = 1, \ldots, \ell$.

Based on Definition~\ref{def:hodlr-eps}, we give a  bound in Lemma~\ref{lemma:approx-diag} on the approximation error in the diagonal blocks $\widetilde{H}_{ii}^{(k)}$, $i=1\colon 2^k$, of the other levels $k=0\colon \ell-1$. Here and in the remainder of the paper we use level $0$ of any HODLR matrix to indicate the HODLR matrix itself, for example, $\widetilde{H}_{11}^{(0)}\equiv \widetilde{H}$.

\begin{lemma}\label{lemma:approx-diag}
Let  $\widetilde{H}$ be a $(\mathcal{T}_{\ell}, p, \varepsilon)$-HODLR matrix associated with $H$. Then for the HODLR matrices $\widetilde{H}^{(k)}_{ii}$, $i=1\colon 2^k$, at level $k \in \{0, \ldots, \ell\}$, 
it holds that 
\begin{equation}\label{eq:approx-diag}
    \|\widetilde{H}^{(k)}_{ii} - {H}^{(k)}_{ii}\|_F \le \varepsilon \|{H}^{(k)}_{ii}\|_F, \quad 
    i=1\colon 2^k.
\end{equation}
\end{lemma}

\begin{proof}
The proof is done by induction. 
Following from Definition~\ref{def:hodlr-eps}, at the final level $\ell$ of the cluster tree,  for  $i \neq j$, we have  
\begin{equation*}
\|\widetilde{H}^{(\ell)}_{ij} - {H}^{(\ell)}_{ij}\| \le \varepsilon \|{H}^{(\ell)}_{ij}\|. 
\end{equation*}
Then, for any diagonal block $H^{(\ell-1)}_{ii}, i=1,\ldots, 2^{\ell-1}$ in level $\ell - 1$, we have
\begin{equation*}
\begin{aligned}
\|\widetilde{H}^{(\ell-1)}_{ii} - {H}^{(\ell-1)}_{ii}\|_F^2 
&= \|\widetilde{H}^{(\ell)}_{2i-1,2i} - {H}^{(\ell)}_{2i-1,2i}\|_F^2 + \|\widetilde{H}^{(\ell)}_{2i,2i-1} - {H}^{(\ell)}_{2i,2i-1}\|_F^2  \\ 
&\le \varepsilon^2 (\|{H}^{(\ell)}_{2i-1,2i}\|_F^2 + \|{H}^{(\ell)}_{2i,2i-1}\|_F^2) 
\le \varepsilon^2 \|{H}^{(\ell-1)}_{ii}\|_F^2,
\end{aligned}
\end{equation*}
where we have exploited the fact that the diagonal blocks $H_{ii}^{(\ell)}, i=1\colon 2^\ell$, at the final level in the approximations $\widetilde{H}$ and $H$ are identical (to that of the original matrix $A$) under the hierarchical partitioning~\eqref{eq:hodlr_block}.

For the inductive step, assume the bound~\eqref{eq:approx-diag} holds for level $k$.
In level $k-1$, similarly, for any diagonal block $H^{(k-1)}_{ii}, i=1,\ldots, 2^{k-1}$, we have
\begin{equation*}
\begin{aligned}
\|\widetilde{H}^{(k-1)}_{ii} - {H}^{(k-1)}_{ii}\|_F^2 
\le& 
\|\widetilde{H}^{(k)}_{2i-1,2i} - {H}^{(k)}_{2i-1,2i}\|_F^2 + \|\widetilde{H}^{(k)}_{2i,2i-1} - {H}^{(k)}_{2i,2i-1}\|_F^2  \\
&+ \|\widetilde{H}^{(k)}_{2i-1,2i-1} - {H}^{(k)}_{2i-1,2i-1}\|_F^2 + \|\widetilde{H}^{(k)}_{2i,2i} - {H}^{(k)}_{2i,2i}\|_F^2 \\ 
\le& \varepsilon^2 (\|{H}^{(k)}_{2i-1,2i}\|_F^2 + \|{H}^{(k)}_{2i,2i-1}\|_F^2 + \|{H}^{(k)}_{2i-1,2i-1}\|_F^2 + \|{H}^{(k)}_{2i,2i}\|_F^2)  \\
=& \varepsilon^2 \|{H}^{(k-1)}_{ii}\|_F^2.
\end{aligned}
\end{equation*}
This shows~\eqref{eq:approx-diag} also holds
for level $k-1$ and therefore completes the proof.
\end{proof}


\section{Mixed-precision construction and representation of HODLR matrices}\label{sec:mixhodlr} 

In this section, we discuss how to exploit mixed precision in the representation and computation of HODLR matrices so as to reduce the computational costs and storage requirements. We start with the representation of the HODLR matrix 
utilizing a set of different precisions.

\subsection{Mixed precision HODLR matrix representation}\label{sec.mp-hodlr}
Our aim is to compress the low-rank blocks and represent them in precisions potentially lower than the working precision, hence obtaining a globally mixed-precision representation of a HODLR matrix while keeping the error in the representation at the same level as an unified working-precision representation.

Recall from Definition~\ref{def:hodlr-eps} that
$\widetilde{H}$ denotes a $(\mathcal{T}_{\ell}, p, \varepsilon)$-HODLR matrix associated with $H$.
Let us assume the off-diagonal blocks from the $k$-th level of $\widetilde{H}$, $1 \le k \le \ell$, are compressed in the form 
\begin{equation}\label{eq:H-offdiag-lowrank}
    \widetilde{H}_{i j}^{(k)}= \widetilde{U}_{i}^{(k)} (\widetilde{V}_{j}^{(k)})^T,
\quad |i-j|=1,
\end{equation}
where $\widetilde{U}_{i}^{(k)}\in \mathbb{R}^{n / 2^k \times p}$ has orthonormal columns to precision $u$ and
$\widetilde{V}_{j}^{(k)} \in \mathbb{R}^{n / 2^k \times p}$. 
The framework for generating $\widehat{H}$, the mixed-precision representation of the HODLR matrix, is presented in Algorithm~\ref{alg:mp-hodlr}, where the HODLR matrix $H$ is partitioned recursively with its off-diagonal blocks at different levels factorized by the low-rank compression~\eqref{eq:H-offdiag-lowrank}.
Note that in Line~\ref{algline:mp-HODLR-low-rank-compute-1} of Algorithm~\ref{alg:mp-hodlr}, we only store the generators $\widehat{U}^{(k+1)}_{2i-1}$ and $\widehat{V}^{(k+1)}_{2i}$ instead of $\widehat{H}^{(k+1)}_{2i-1,2i}$ itself in practice (the same applies to Line~\ref{algline:mp-HODLR-low-rank-compute-2} in Algorithm~\ref{alg:mp-hodlr}, and Lines~\ref{algline:adap-HODLR-low-rank-compute-1}--\ref{algline:adap-HODLR-low-rank-compute-2} of Algorithm~\ref{alg:mp-hodlr-adap}).
In addition to the working precision $u$, the algorithm involves $\ell$ other floating-point precisions that are associated with different levels of low-rank compression of the off-diagonal blocks.
For simplicity of exposition, the intermediate outputs $\widetilde{H}^{(k)}_{ij}$, $1 \le k \le \ell$, that are computed in the process, are not presented explicitly. 

\begin{algorithm}
\caption{Mixed-precision HODLR Format ($H;\ \mathcal{T}_{\ell}; \{u, u_1, u_2, \ldots, u_\ell\}$)}
\label{alg:mp-hodlr}
\begin{algorithmic}[1]
 \STATE{$k \leftarrow 0, H\leftarrow H^{(0)}$}
    \WHILE{$k < \ell$}
    \FOR{$i :=1\colon 2^k$}    
    \STATE{Partition $H_{ii}^{(k)}$ by \eqref{eq:hodlr_block}}
    \STATE{$\widehat{H}^{(k+1)}_{2i-1,2i} \leftarrow \widehat{U}^{(k+1)}_{2i-1}(\widehat{V}^{(k+1)}_{2i})^T$ \textcolor{blue}{\small [Compute in $u$, and store $\widehat{U}^{(k+1)}_{2i-1}, \widehat{V}^{(k+1)}_{2i}$ in $u_{k+1}$]}}\label{algline:mp-HODLR-low-rank-compute-1}
    \STATE{$\widehat{H}^{(k+1)}_{2i,2i-1} \leftarrow \widehat{U}^{(k+1)}_{2i} (\widehat{V}^{(k+1)}_{2i-1})^T$ \textcolor{blue} {\small [Compute in $u$, and store $\widehat{U}^{(k+1)}_{2i}, \widehat{V}^{(k+1)}_{2i-1}$ in $u_{k+1}$]}}\label{algline:mp-HODLR-low-rank-compute-2} 
    \ENDFOR
    \STATE{$k \leftarrow k+1$}
    \ENDWHILE
    \FOR{$i :=1\colon 2^{\ell}$}  
    \STATE{$\widehat{H}_{i, i}^{(\ell)} \leftarrow H_{i, i}^{(\ell)}$ \textcolor{blue}{\small [Store $\widehat{H}_{i, i}^{(\ell)}$ in $u$]}}
    \ENDFOR
\RETURN $\widehat{H}$
\end{algorithmic}
\end{algorithm}

Since the approximate ranks of the off-diagonal blocks $\widetilde{H}^{(k)}_{ij}$, $1 \le k \le \ell$ are generally not known a priori, in the actual implementation of Algorithm~\ref{alg:mp-hodlr} the factorizations are typically computed with the threshold $\varepsilon>u$ (since the factorizations are calculated in the working precision $u$) as an input parameter.
Also, we are analysing the case where the factorizations are computed in the working precision $u$, yet our analysis can be easily extended to use a different precision and this will only affect the attainable computational accuracy $\varepsilon$.

Now we state the main result which bounds the global error introduced by the mixed-precision HODLR compression via Algorithm~\ref{alg:mp-hodlr} and hence
indicates how to choose the precisions $u_k,\ k=1\colon \ell$, in order to keep the error at a satisfying level.

\begin{theorem}[Global error in mixed-precision HODLR representation]\label{thm:mpblr-uv}
Let $\widetilde{H}$ be a $(\mathcal{T}_{\ell}, p, \varepsilon)$-HODLR matrix associated with the HODLR matrix $H$, and denote by $\widehat{H}$ the mixed-precision representation produced by Algorithm~\ref{alg:mp-hodlr}. 
Define
$
\xi_k := 
\max_{|i-j|=1}\norm{\widetilde{H}^{(k)}_{ij}}_F/\norm{\widetilde{H}}_F,  1 \le k \le \ell,
$
where $\widetilde{H}^{(k)}_{ij}, |i-j|=1$ denotes any off-diagonal block 
from the $k$-th level. Then it holds that  
\begin{equation}\label{eq:global-err-bnd}
\textstyle
\norm{H-\widehat{H}}_F \lesssim  \left(2\sqrt{2}\big(\sum_{k=1}^{\ell}
2^{k} \xi_k^{2} u_k^2\big)^{\frac{1}{2}} + \varepsilon\right)\norm{H}_F,
\end{equation}
and, furthermore, if the precisions are chosen by 
\begin{equation}\label{eq:prec_choose}
    u_k\leq \varepsilon/ (2^{k/2}\xi_{k}),
\end{equation}

the bound becomes
$
\norm{H-\widehat{H}}_F \lesssim (2\sqrt{2\ell}+1)\varepsilon\norm{H}_F.   
$
\end{theorem}

\begin{proof}
Applying~\cite[Lem.~2.2]{abbg22} to $\widetilde{H}_{ij}^{(k)}$, $1 \le k \le \ell$, and using the relation 
$\norm{\widetilde{H}^{(k)}_{ij}}_F \le \xi_k \norm{\widetilde{H}}_F$, we obtain
\begin{equation}\label{eq:diff-tilde-hat-offdiag}
\|\widetilde{H}^{(k)}_{i j} - \widehat{H}^{(k)}_{i j}\|_F 
    \le (2+ \sqrt{p} u_k) u_k \|\widetilde{H}^{(k)}_{i j}\|_F
    \le \xi_k (2+ \sqrt{p} u_k) u_k \norm{\widetilde{H}}_F.
\end{equation}
Then it follows
\begin{equation*}
\begin{aligned}
    \norm{\widetilde{H} - \widehat{H}}^2_F &= \textstyle
    \sum_{k=1}^{\ell} \sum_{i=1}^{2^k} 
    \left( 
\|\widetilde{H}_{2i-1,2i}^{(k)} - \widehat{H}_{2i-1,2i}^{(k)}\|^2_F  + \|\widetilde{H}_{2i,2i-1}^{(k)} - \widehat{H}_{2i,2i-1}^{(k)}\|^2_F  \right)\\& \textstyle \le
\left(2 \sum_{k=1}^{\ell}
2^{k} \xi^{2}_k \left( 2+\sqrt{p}u_k\right)^2 u_k^2 \right)\norm{\widetilde{H}}^2_F
\lesssim 8\left(\sum_{k=1}^{\ell}
2^{k} \xi^{2}_k u_k^2 \right)\norm{\widetilde{H}}^2_F,
\end{aligned}
\end{equation*}
where we have ignored higher order terms in $u_k$ and used the fact resulting from Algorithm~\ref{alg:mp-hodlr} 
that the diagonal blocks 
$\widehat{H}_{ii}^{(\ell)}$ and $\widetilde{H}_{ii}^{(\ell)}$,
$i=1\colon 2^\ell$, at the final level are identical to that of $H$.
Then~\eqref{eq:global-err-bnd} straightforwardly follows by substituting the inequality $\normF{\widetilde{H} - H}\le \varepsilon\normF{H}$
from Lemma~\ref{lemma:approx-diag} into the bound
$\|\widehat{H}-H\|_F\le \|\widehat{H}-\widetilde{H}\|_F + \|\widetilde{H}-H\|_F$.

Now if the precisions are set by $u_k\leq \varepsilon/(2^{k/2}\xi_{k})$, 
we have
$\sum_{k=1}^{\ell} 2^{k} \xi^{2}_k u_k^2  \leq  \ell \varepsilon^2$
and the result follows.
\end{proof}

Theorem~\ref{thm:mpblr-uv} illustrates the fact that the global error in the mixed-precision HODLR representation is proportional to the depth $\ell$ of the hierarchical cluster tree, given that the $\varepsilon$ is chosen to balance both the errors from the low-rank approximation to the off-diagonal blocks and the low-precision conversion of the low-rank factors. 
The expression~\eqref{eq:prec_choose} for admissible precision $u_k$ has two important implications.
First, it indicates that $u_k$ depends inversely proportionally on $\xi_k$, which essentially characterizes the relative importance of the off-diagonal blocks in level-$k$ to the whole matrix in terms of magnitude. 
Second, the choice of $u_k$ should also be adaptive to the level index $k$,
to take into account the error introduced by an increased level of the hierarchical cluster tree.

Finally, from~\eqref{eq:prec_choose} we know for two successive precisions $u_k$ and $u_{k-1}$, $u_k\ge u_{k-1}$ holds if 
$$
\xi_{k-1} \ge \sqrt{2}\xi_k, \quad k=2\colon \ell.
$$
Since the size of 
$\widetilde{H}^{(k)}_{ij}$ is half of the size of 
$\widetilde{H}^{(k-1)}_{ij}$, this condition can generally be expected to hold for HODLR matrices that are not too badly-scaled across off-diagonal blocks of different levels.
Therefore, for such matrices we can generally expect the  chosen precisions to satisfy
$u \le u_1 \le u_2 \cdots \le u_\ell$.

Each diagonal block of a HODLR matrix can be formulated as a HODLR matrix itself; we thus have the following corollary of Theorem~\ref{thm:mpblr-uv} that bounds the error on the diagonal blocks of different levels
resulting from the mixed-precision representation.

\begin{corollary}[Local error of mixed-precision HODLR representation]\label{thm:mpblr-uv-local}
Let $\widetilde{H}$ be a $(\mathcal{T}_{\ell}, p, \varepsilon)$-HODLR matrix associated with the HODLR matrix $H$, and let $\widehat{H}$ denote the mixed-precision representation produced by Algorithm~\ref{alg:mp-hodlr}. 
With the same $\xi_k$ defined in Theorem~\ref{thm:mpblr-uv} and 
the precisions $u_{k+\ell'}$, $k=1\colon \ell -\ell'$, chosen by $u_{k+\ell'}\leq \varepsilon/(2^{k/2}\xi_{k+\ell'})$,
each diagonal block $\widehat{H}^{(\ell')}_{ii}$, $i=1\colon 2^{\ell'}$ in the level $\ell'$ of $\widehat{H}$ satisfies, for $\ell'=1\colon \ell$, 

$$
\norm{\widetilde{H}^{(\ell')}_{ii} - \widehat{H}^{(\ell')}_{ii}}_F  
\lesssim 
2\sqrt{2(\ell-\ell')} \varepsilon \norm{H_{ii}^{(\ell')}}_F
\approx 
2\sqrt{2(\ell-\ell')} \varepsilon \norm{\widetilde{H}_{ii}^{(\ell')}}_F
$$
and
$$
\norm{H^{(\ell')}_{ii}-\widehat{H}^{(\ell')}_{ii}}_F \lesssim (2\sqrt{2(\ell-\ell')} + 1)\varepsilon\norm{H_{ii}^{(\ell')}}_F.
$$
\end{corollary}

\subsection{An adaptive-precision algorithm}
Since $0<\xi_k< 1$ holds for $k=1\colon \ell$ (the upper bound should not be attained for any meaningful HODLR format), 
\new{together with the constrain $\varepsilon>u$, the bound~\eqref{eq:prec_choose} implies that generally no higher-than-working precisions are needed among $u_k$ for a HODLR matrix with mild depth $\ell$, say, $\ell\le 10$ (so $2^{k/2}\le 32$); 
this statement is very pessimistic since the threshold $\varepsilon$ in most cases should be at least several orders of magnitude larger than the working precision $u$. 
Admittedly, the precision $u_k$ can be higher than the working precision $u$ when the maximal level $\ell$ becomes larger than a certain threshold.
In this case, an interesting question that arises naturally is, at which level of a HODLR matrix the working precision (e.g., fp32 or fp64) may be exceeded. However, an answer cannot be provided
without further assumptions on $\xi_k$, or the structure of the underlying matrix and the approximation threshold $\varepsilon$.
} 

\begin{figure}[t]
\centering
\subfigure[saylr3]{
\includegraphics[width=0.35\textwidth]{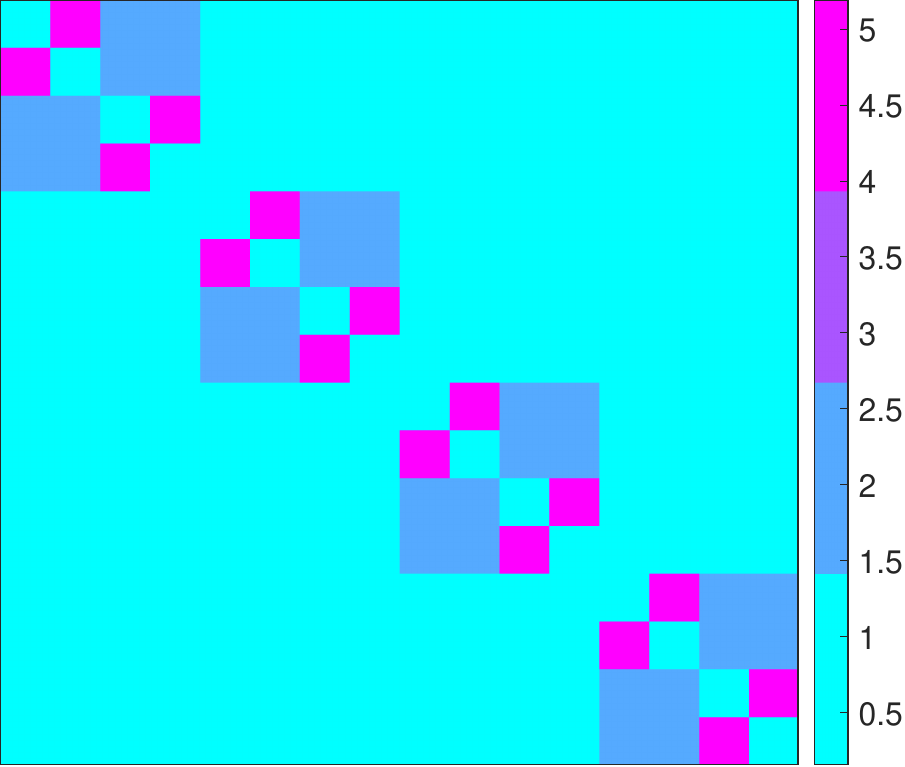}}
\hspace{10pt}
\subfigure[LeGresley\_2508\qquad]{\includegraphics[width=0.364\textwidth]{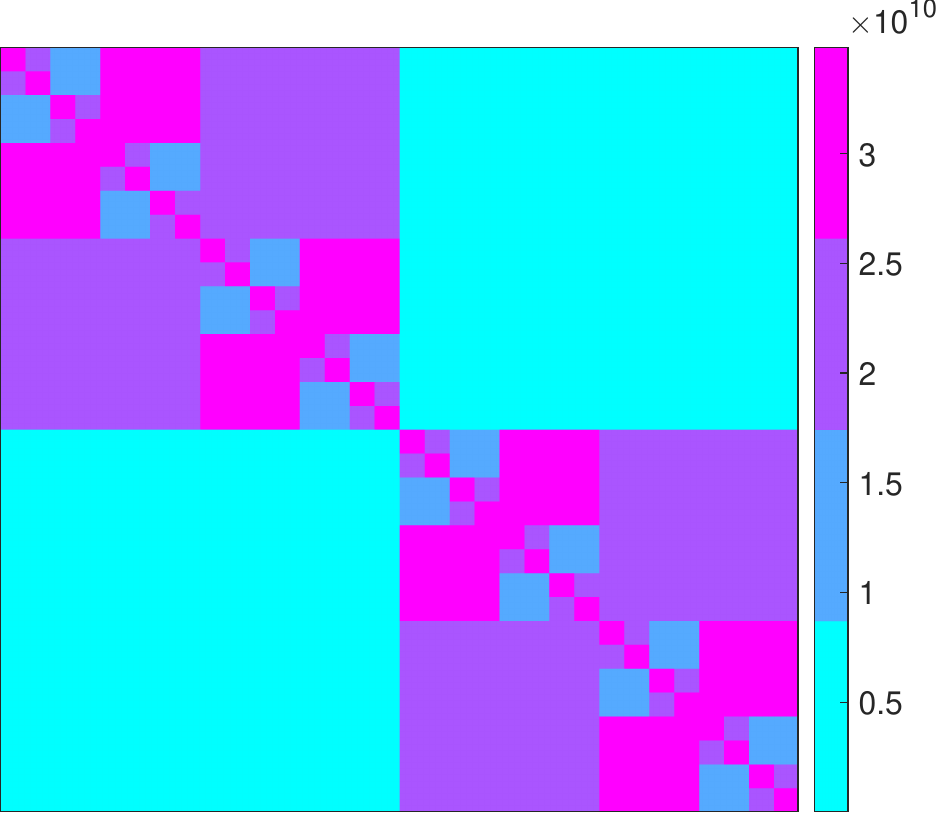}}
\caption{Different matrices often show various norm distributions among layers. We compute the maximum off-diagonal block norm for the HODLR matrix with depth $\ell=6$ and the color indicates the size of the norm. According to the norm, using $\varepsilon=10^{-4}$ and the set of precisions $\{q52, bf16, fp16, fp32, fp64\}$, the algorithm chooses  $\{\text{bf16, fp16, fp16, fp32, fp32, fp32}\}$ for the precision used for each layer for the matrix saylr3 while $\{\text{q52, fp32, fp32, fp32, fp32, fp32}\}$ is chosen for the matrix LeGresley\_2508.}
\label{fig:normdist}
\end{figure}

Based on Theorem~\ref{thm:mpblr-uv}, we design an adaptive-precision HODLR representation scheme, presented in Algorithm~\ref{alg:mp-hodlr-adap}, where we choose $u_k \leq \varepsilon \big/ (2^{k/2}\xi_k)$ so as to achieve 
\begin{equation}\label{eq:error-adap-prec-represent}
\norm{H-\widehat{H}}_F \lesssim (2\sqrt{2\ell} + 1) \varepsilon\norm{H}_F.
\end{equation}
With this choice of $u_k$, we have, from~\eqref{eq:diff-tilde-hat-offdiag}, for $k=1\colon \ell$ and $i\neq j$,
\begin{equation*}
\|\widetilde{H}^{(k)}_{i j} - \widehat{H}^{(k)}_{i j}\|_F 
    \lesssim 2\xi_k u_k \norm{\widetilde{H}}_F
    = 2^{1- k/2}\varepsilon \norm{\widetilde{H}}_F
    \approx 2^{1- k/2}\varepsilon \norm{H}_F,
\end{equation*}
and therefore
\begin{align}\label{eq:error-adap-prec-offdiagbloc}
\norm{H_{ij}^{(k)}-\widehat{H}_{ij}^{(k)}}_F 
&\le \norm{\widetilde{H}_{ij}^{(k)}-\widehat{H}_{ij}^{(k)}}_F +
\norm{H_{ij}^{(k)}-\widetilde{H}_{ij}^{(k)}}_F 
\lesssim (2^{1-k/2} + 1) \varepsilon\norm{H}_F.
\end{align}
Also, from Corollary~\ref{thm:mpblr-uv-local} we have
\begin{equation}\label{eq:error-adap-prec-diagbloc}
\norm{H_{ii}^{(\ell)}-\widehat{H}_{ii}^{(\ell)}}_F \lesssim \varepsilon\norm{H_{ii}^{(\ell)}}_F,\quad 
i=1\colon 2^{\ell},
\end{equation}
which can also be seen from the fact that we are storing these diagonal blocks of level $\ell$ in the working precision $u$, so
$\widetilde{H}^{(\ell)}_{ii} = \widehat{H}^{(\ell)}_{ii}$,  $i=1\colon 2^{\ell}$.

\begin{algorithm}[t]
\caption{Adaptive-precision HODLR Format ($H; \mathcal{T}_{\ell}; u; \varepsilon; \mathcal{U}=\{u_1, \ldots, u_{N}\}$)}
\label{alg:mp-hodlr-adap}
\begin{algorithmic}[1]

\STATE{$k=0, H^{(0)}_{11} \leftarrow H$}
\STATE{\texttt{nrmAll} $\leftarrow$ Compute $\|H\|_F^2$}

\WHILE{$k < \ell$}
\FOR{$i :=1\colon 2^{k}$}    
\STATE{Partition $H_{ii}^{(k)}$ by \eqref{eq:hodlr_block}}
\ENDFOR

\STATE{\texttt{nrm} $\leftarrow$ $\max_{|i-j|=1} \|H_{ij}^{(k+1)}\|_F^2$}

\STATE{$\xi_{k+1} =\sqrt{\texttt{nrm} / \texttt{nrmAll}}$}

\STATE{$u_{k+1}' \leftarrow \varepsilon \big/    (2^{(k+1)/2}\xi_{k+1})$}

\STATE{$u_{k+1} \leftarrow \max\{u^* \le u_{k+1}': u^*\in 
\{u\} \cup \mathcal{U} \}$}

\FOR{$i :=1\colon 2^{k}$}    

\STATE{$\widehat{H}^{(k+1)}_{2i-1,2i} \leftarrow \widehat{U}^{(k+1)}_{2i-1}(\widehat{V}^{(k+1)}_{2i})^T$ \textcolor{blue}{\small [Compute in $u$, and store $\widehat{U}^{(k+1)}_{2i-1}, \widehat{V}^{(k+1)}_{2i}$ in $u_{k+1}$]}} \label{algline:adap-HODLR-low-rank-compute-1}

\STATE{$\widehat{H}^{(k+1)}_{2i,2i-1} \leftarrow \widehat{U}^{(k+1)}_{2i} (\widehat{V}^{(k+1)}_{2i-1})^T$ \textcolor{blue} {\small [Compute in $u$, and store $\widehat{U}^{(k+1)}_{2i}, \widehat{V}^{(k+1)}_{2i-1}$ in $u_{k+1}$]}}  \label{algline:adap-HODLR-low-rank-compute-2}
\ENDFOR
\ENDWHILE
\FOR{$i:=1\colon 2^{\ell}$}  
\STATE{$\widehat{H}_{i, i}^{(\ell)} \leftarrow H_{i, i}^{(\ell)}$ \textcolor{blue}{\small [Store $\widehat{H}_{i, i}^{(\ell)}$ in $u$]}}
\ENDFOR
\RETURN $\widehat{H}$
\end{algorithmic}
\end{algorithm}

In addition to the HODLR matrix $H$ associated with a binary tree 
$\mathcal{T}_{\ell}$, the working precision $u$, and $\varepsilon$, the algorithm also takes a set $\mathcal{U}=\{u_1, u_2, \ldots, u_{N}\}$ of available precisions as the input, so the algorithm will calculate the feasible precision $u_k'$ by~\eqref{eq:prec_choose} and then take the precision $u_k$
from $\mathcal{U}$ which is the closest lower precision to $u_k'$.
Note that the number $N$ of the available precisions is not necessarily equal to $\ell$. 
The algorithm estimates the $\xi_k$ in Theorem~\ref{thm:mpblr-uv}
using the norms of $H$ and its
off-diagonal blocks from level $k$; this only introduces a
relative error of order $\varepsilon$ in the estimation of the norms. An illustrative example is shown in \figurename~\ref{fig:normdist}. 

\new{The norm estimations in general take $O(n^2)$ flops,
but this overhead is reduced when the matrix $H\equiv A$ is in a sparse or structured format. 
This computational cost can also be substantially mitigated by the fact that
the estimations can be rough and be carried out in a precision lower than the working precision
since it is enough to have $u_k$ to the correct order of magnitude. It may also be possible to use randomized techniques; see, e.g., \cite{bukr21}. 
Generally, we assume that the construction of such a mixed-precision HODLR matrix can be done in a precomputation stage, and the format can be reused for multiple subsequent computational kernels such that the computational cost is amortized. 
}

\subsection{Matrix--vector products}
\new{As previously mentioned,}
using the HODLR representation of a matrix $A \in \mathbb{R}^{n \times n}$, denoted by $H$, matrix--vector products can be efficiently performed with complexity $O(p n \log n)$ where $p$ is the maximum rank of all off-diagonal blocks. 
The algorithm for computing the matrix--vector product with a HODLR matrix is described in Algorithm~\ref{alg:matvecprod}.
\new{Note that if the working precision of the computational kernel, e.g., the matrix--vector product and LU factorization, is higher than some storage formats used in the adaptive-precision HODLR matrix, we can still retain memory and communication savings by casting the lower-precision formats to the working precision on the fly during computations.}

\begin{algorithm}[t]
\caption{HODLR\_MATVEC($H; x; \mathcal{T}_{\ell}$)}
\label{alg:matvecprod}
\begin{algorithmic}[1]
\STATE{$b  \leftarrow \mathbf{0} \in \mathbb{R}^{n}$}
\FOR{$k = 1 \colon \ell$}
\STATE{Partition $b$ into $2^k$ blocks, i.e., $(b_{i}^{(k)})_{i=1}^{2^k}$
\STATE{Partition $x$ into $2^k$ blocks, i.e., $(x_{i}^{(k)})_{i=1}^{2^k}$
}
}
\FOR{$i=1\colon 2^{k-1}$}    

\STATE{$b_{2i-1}^{(k)} \leftarrow b_{2i-1}^{(k)} + U^{(k)}_{2i-1}(V^{(k)}_{2i})^Tx_{2i}^{(k)}$ \textcolor{blue}{\small [Compute in precision $u$]}} 
\STATE{$b_{2i}^{(k)} \leftarrow b_{2i}^{(k)} + U^{(k)}_{2i} (V^{(k)}_{2i-1})^Tx_{2i-1}^{(k)}$ \textcolor{blue}{\small [Compute in precision $u$]}} 
\ENDFOR

\ENDFOR
\FOR{$i=1\colon 2^\ell$}  
\STATE{$b_{i}^{(\ell)} \leftarrow b_{i}^{(\ell)} + H^{(\ell)}_{i, i}x_{i}^{(\ell)}$ \textcolor{blue}{\small [Compute in precision $u$]}} 
\ENDFOR
\RETURN $b$
\end{algorithmic}
\end{algorithm}

We first give error bounds on the working precision $u$ so that the backward error in computing the matrix--vector product 
in finite precision does not exceed the error resulting from 
inexact representation of the matrix.
The bounds are applicable to a single matrix that is approximated (possibly via some low-rank truncation and low-precision representation), and then it will be utilized to bound the backward error in the HODLR matrix--vector product.
The key idea is that,
if the HODLR matrix $H$ is approximated by the mixed-precision representation $\widehat{H}$, to calculate the matrix--vector product $b\leftarrow \widehat{H}x$ we should try to balance the errors occurring in the approximation of $\widehat{H}$ and in the finite-precision computation.

\new{In the proofs of the following Lemma~\ref{lemma:msvd} and Lemma~\ref{lemma:mqr}, we assume $\gamma_n\approx nu$, which only holds if $n\ll 1/u$. This can be violated for large-scale problems and low precisions (e.g., for $n\ge 2050$ for fp16 format and for $n\ge 256$ for bf16 format), although this aligns with the usage of these low-precision formats for smaller HODLR blocks.
}

\begin{lemma}\label{lemma:msvd}
Let $A_{p}=X_p \Sigma_p Y_p^T= \sum_{i=1}^p  \sigma_i x_i y_i^T$ be the best rank-$p$ approximation of $A$.

Then the error due to finite precision computation of $\widehat{y} = \mathrm{fl}(A_{p} x)$ will be no larger than the error due to low-rank approximation when the working precision has unit roundoff
$u \le \sigma_{p+1}/(pn \sigma_1)$.

\end{lemma}

\begin{proof}
Since $A_p$ is the best rank-$p$ approximation to $A$, we have
$\sigma_{p+1} =\normt{A - A_{p}} \equiv \normt{\Delta A_L}$, 
where $\Delta A_L$ represents the error caused by low-rank approximation of $A$.

By standard rounding error analysis~\cite[sect.~3.5]{high:ASNA2}
\begin{equation*}
    \widehat{y} = \mathrm{fl}(A_{p} x) = (A_{p} + \Delta A_F) x, \quad 
    |\Delta A_F| \le \gamma_n |A_{p}|,
\end{equation*}
where $\Delta A_F$ represents the error due to finite precision computation. Specified for the $2$-norm, this elementwise bound becomes~\cite[Lem.~6.6]{high:ASNA2}
$\normt{\Delta A_F} \le \rank(A_{p})\gamma_n \normt{A_{p}}$.
Together, we have
\begin{equation*}
\begin{aligned}
\widehat{y} = \mathrm{fl}(A_{p} x) = (A_{p} + \Delta A_F)x = (A - \Delta A_L + \Delta A_F) x \equiv (A + \Delta A)x,
\end{aligned}
\end{equation*}
where $\Delta A =  \Delta A_F - \Delta A_L$, and hence
$\normt{\Delta A}\le \normt{\Delta A_F} + \normt{\Delta A_L}$.

Taking the approximation $\gamma_n \approx nu$, we have
\begin{equation*}
    \normt{\Delta A_F} - \normt{\Delta A_L} 
    \le pnu\normt{A_{p}} - \normt{\Delta A_L}  
    = pnu \sigma_1 - \sigma_{p+1},
\end{equation*}
It follows, when $u \le \sigma_{p+1}/(pn \sigma_1)$,
the error due to finite precision computation will be no larger than the error due to low-rank approximation, i.e.,
$\normt{\Delta A_F} \le \normt{\Delta A_L}$.
\end{proof}

As long as $u$ is below the threshold given in Lemma~\ref{lemma:msvd}, the error due to finite precision computation will not significantly affect the accuracy of the matrix--vector product relative to the accuracy lost due to intentional approximation. Intuitively, this means that the more inexact the low-rank representation, 
the lower the precision we can safely use.

In practice, the approximation $A_p$ in Lemma~\ref{lemma:msvd} is usually obtained via a truncated SVD that is not constrained by the rank of $A_p$ but by a truncation threshold parameter $\eta>0$ such that
$\|A - A_{p}\|_2 \le \eta$. In this case, the rank of $A_p$, which we denote by $p(\eta)$, is not necessarily $p$ but depends on $\eta$.
In the following, we prove in a similar way an analogous result
with the more practical tolerance $\eta$, where the matrix approximant can come from any rank-revealing factorization, not necessarily a truncated SVD. The result is proven in the Frobenius norm to facilitate our use case.

\begin{lemma}\label{lemma:mqr}
Let $\widehat{A}_{p}$ an approximation of $A$ such that $\|A - \widehat{A}_{p}\|_F \approx \eta$ for some $\eta>0$.
Then the error due to finite precision computation of $\widehat{y} = \mathrm{fl}(\widehat{A}_{p} x)$ will be no larger than the error due to the computed inexact representation when the working precision has unit roundoff $u \le \eta/(n\normF{\widehat{A}_{p}})$.
\end{lemma}

\begin{proof}
By defining $\Delta A_L \equiv  A - \widehat{A}_p$, we have
$\|\Delta A_L\|_F \approx \eta$,
where $\Delta A_L$ represents the accuracy lost due to intentional approximation of $A$ and the finite precision computation of this approximation.
On the other hand, we have, by standard rounding error analysis~\cite[sect.~3.5]{high:ASNA2}
\begin{equation*}
    \widehat{y} = \mathrm{fl}(\widehat{A}_{p} x) = (\widehat{A}_{p} + \Delta A_F) x, \quad 
    |\Delta A_F| \le \gamma_n |\widehat{A}_{p}|,
\end{equation*}
where $\Delta A_F$ represents the error due to finite precision computation of the matrix--vector product. The matrix Frobenius norm, which corresponds to the vector $2$-norm on $\mathrm{R}^{n^2}$, is obviously absolute and thus is monotone~\cite[Thm.~6.2]{high:ASNA2}, and so from the
elementwise bound we have
\begin{equation*}
\normF{\Delta A_F} = \normF{|\Delta A_F|}\le \gamma_n \normF{|\widehat{A}_{p}|} = \gamma_n \normF{\widehat{A}_{p}}.
\end{equation*}
Together, we have
\begin{equation*}
\begin{aligned}
\widehat{y} = \mathrm{fl}(\widehat{A}_{p} x) = (\widehat{A}_{p} + \Delta A_F) x = (A - \Delta A_L + \Delta A_F) x \equiv (A + \Delta A)x,
\end{aligned}
\end{equation*}
where $\Delta A =  \Delta A_F - \Delta A_L$, and hence
$\normF{\Delta A}\le \normF{\Delta A_F} + \normF{\Delta A_L}$.

Taking the approximations $\gamma_n \approx nu$, we have
\begin{equation*}
\begin{aligned}
    \normF{\Delta A_F} - \normF{\Delta A_L} 
    & \le nu \normF{\widehat{A}_{p}} - \normF{\Delta A_L}
    \approx nu \normF{\widehat{A}_{p}} - \eta.
\end{aligned}
\end{equation*}
Then it follows that when $u \le \eta/ (n\normF{\widehat{A}_{p}})$,
the error due to the finite precision computation of the matrix--vector product will be no larger than the error due to the computed inexact representation, i.e., 
$\normF{\Delta A_F} \le \normF{\Delta A_L}$.
\end{proof}

Finally, we are ready to present a overall bound for the HODLR matrix--vector product in Algorithm~\ref{alg:matvecprod} with the matrix represented in the mixed-precision HODLR format.

\begin{theorem}\label{thm:mat-vec-bwerr}
    Let $\widetilde{H}$ be a $(\mathcal{T}_{\ell}, p, \varepsilon)$-HODLR matrix associated with the HODLR matrix $H$, and let $\widehat{H}$ denote the mixed-precision representation produced by Algorithm~\ref{alg:mp-hodlr-adap}. 
    If $b = \widehat{H} x$ is computed via Algorithm~\ref{alg:matvecprod} in a working 
    precision $u \le \varepsilon/n$,
    then the computed $\widehat{b}$ satisfies
    \begin{align*}
        \widehat{b} = \mathrm{fl}(\widehat{H} x) = (H + \Delta H) x, \quad 
    \normF{\Delta H} &\le 2(\sqrt{2} + 1) 
    \sqrt{2^{\ell+1}+2^{\ell-1}} \varepsilon \normF{H}. 
    \end{align*} 
\end{theorem}

\begin{proof}
\new{Assuming the representation error in one off-diagonal block $\widehat{H}^{(k)}_{ij}$ of $\widehat{H}$ attains its upper bound~\eqref{eq:error-adap-prec-offdiagbloc}, we know from Lemma~\ref{lemma:mqr} that the error in computing the matrix--vector product associated with $\widehat{H}^{(k)}_{ij}$
will not exceed the error in computing the approximation $\widehat{H}^{(k)}_{ij}$ if the working precision has unit roundoff
    \begin{equation*}
    u \le \frac{\varepsilon}{n} 
    = \frac{\varepsilon\normF{H}}{n\normF{H} } \leq
    \frac{(2^{1-k/2}+1)\varepsilon\normF{H}}
    {n\normF{H^{(k)}_{ij}}} \approx
    \frac{(2^{1-k/2}+1)\varepsilon\normF{H}}
    {n\normF{\widehat{H}^{(k)}_{ij}} }.
\end{equation*}
This means if we choose 
$u \le \varepsilon/n$,
then for \textit{any} off-diagonal block $\widehat{H}^{(k)}_{ij}$, $k=1\colon \ell$, the error in computing the matrix--vector product associated with the matrix
will not exceed the maximum possible error in its computed approximation.} Note that this bound for the working precision holds independent of the vector that is being multiplied.

For the diagonal blocks in level $\ell$, from~\eqref{eq:error-adap-prec-diagbloc} we have the bound
\begin{equation}\label{eq:error-adap-prec-diagbloc-loose} 
\norm{H_{ii}^{(\ell)}-\widehat{H}_{ii}^{(\ell)}}_F \lesssim \varepsilon\norm{H_{ii}^{(\ell)}}_F
\le \varepsilon\norm{H}_F,
\quad i=1\colon 2^{\ell},
\end{equation}
which, by using Lemma~\ref{lemma:mqr}, leads to the same bound $u \le \varepsilon/n$ for the working precision at which the matrix--vector products associated with the diagonal blocks should be performed.

Define ${H}_L = \widehat{H} - H$,
where the off-diagonal blocks of ${H}_L$ are formed by 
$(H_L)_{ij}^{(k)} = \widehat{H}_{ij}^{(k)} - H_{ij}^{(k)}$ for 
$i\ne j$ with $k=1 \colon\ell$ and $i=j$ with $k=\ell$. 
Since each diagonal and off-diagonal block in $\widehat{H}$ will be multiplied by a segment of vector $x$ exactly once in forming $\widehat{H} x$, by ignoring the errors in the summation of the vector elements (which are usually negligible compared with the error in the block matrix--vector products), with $u \le \varepsilon/n$, we have
\begin{equation}\label{eq:matvec-proof-bound}
        \widehat{b} = \mathrm{fl}(\widehat{H} x) 
        = (\widehat{H} + H_F) x 
        = (H + {H}_L + H_F) x, \quad 
    \normF{(H_F)_{ij}^{(k)}} \le  \normF{(H_L)_{ij}^{(k)}},
\end{equation} 
because in this case the backward error in any matrix--vector product performed in forming $\widehat{H} x$ via Algorithm~\ref{alg:matvecprod} is bounded above by the approximation error in the associated matrix.
We are thus interested in a bound for $\normF{{H}_L}^2$, which, by using~\eqref{eq:error-adap-prec-offdiagbloc} and~\eqref{eq:error-adap-prec-diagbloc-loose}, is
\begin{align*}
\normF{{H}_L}^2 &= \textstyle
\sum_{1\le k\le \ell, i\ne j} \normF{\widehat{H}_{ij}^{(k)} - H_{ij}^{(k)}}^2 + \sum_{i=1}^{2^{\ell}} \normF{\widehat{H}_{ii}^{(\ell)} - H_{ii}^{(\ell)}}^2\\
&\lesssim \textstyle
\sum_{k=1}^{\ell} 2^k\left( (2^{1-k/2} + 1) \varepsilon\norm{H}_F
\right)^2  + 2^\ell\cdot \varepsilon^2\norm{H}_F^2 \\
&\le  \textstyle \left(
\sum_{k=1}^{\ell} 2^k(\sqrt{2} + 1)^2  + 2^\ell  \right)\varepsilon^2 \normF{H}^2 
\le (\sqrt{2} + 1)^2 \left( 2^{\ell+1} + 2^{\ell-1}  \right)\varepsilon^2 \normF{H}^2.
\end{align*}
From this bound and~\eqref{eq:matvec-proof-bound} we have
$$
\normF{H_F} \le \normF{H_L} \le 
(\sqrt{2} + 1) \sqrt{2^{\ell+1} + 2^{\ell-1}}
\varepsilon \normF{H},
$$
and therefore the desired backward error bound follows by using $\normF{\Delta H} \le \normF{H_L} + \normF{H_F}$.
\end{proof}

Theorem~\ref{thm:mat-vec-bwerr} gives a bound on the working precision 
$u$ such that the backward error resulting from
computing the matrix--vector product in finite precision does not outweigh the approximation error in the HODLR matrix $\widehat{H}$ that is stored in a mixed-precision fashion. Since the precisions at which different levels of the matrix are stored are carefully chosen to maintain a satisfying normwise error, the overall backward error in the matrix--vector product 
$\widehat{H}x$ is also well bounded, at the level of 
$O(2^{\ell/2}\varepsilon \normF{H}$).

\subsection{LU factorization}
The HODLR format has been widely used for approximating the  discretized operators of differential equations, and the underlying application often involves solving linear systems. This thus involves the computation of the LU factorization of a HODLR matrix followed by forward and backward subsitutions. Alternatively, iterative solvers can also be employed for the linear system, in which case a good preconditioner is often required for fast convergence. Such preconditioners can be obtained also via a HODLR-based LU factorization.
Compared with the traditional dense LU factorization, which requires $O(n^3)$ flops, the HODLR-based LU factorization
can reduce the complexity to approximately no more than $O(p^3n\log n+ p^2 n\log^2n)$~\cite{bakr16}.

In this section we derive error bounds on the LU factorization of $\widehat{H}$, the matrix represented in the mixed-precision HODLR format.
In Algorithm~\ref{alg:hmlu} we describe a recursive algorithm for computing an LU factorization of a HODLR matrix. Each recursive call inside the algorithm essentially consists of four main steps to produce the block LU factorization
\begin{equation*}
    \begin{bmatrix}
       H_{11} &  H_{12} \\
    H_{21} &  H_{22}
    \end{bmatrix} = \begin{bmatrix}
    L_{11} & \\
    L_{21} &  L_{22} 
\end{bmatrix}
\begin{bmatrix}
    U_{11} & U_{12} \\
    &  U_{22} 
\end{bmatrix},
\end{equation*}
where $L_{11}$ and $L_{22}$ are lower triangular and $U_{11}$ and $U_{22}$ are upper triangular.
The algorithm first computes an LU factorization $H_{11}=L_{11}U_{11}$, which is on a HODLR matrix of half the size; Then $U_{12}=L_{11}^{-1}H_{12}$ and 
$L_{21}=H_{21}U_{11}^{-1}$ are solved in a recursive way, which exploit the block structure of the matrices recursively to reduce the problem to the size of the matrices on the lowest level.

Finally, in order to get an LU factortization of the Schur complement $H_{22}-L_{21}U_{12}$, HODLR matrix addition is performed following a $\mathcal{H}$-matrix multiplication~\cite{BORM2003405}, 
where a suitable HODLR rank truncation might be needed for the Schur complement.
Each of these four steps involves subproblems of half the size, and on the bottom-level dense routines, including dense LU factorization and dense triangular solves, are invoked. The procedure is inherently sequential, meaning that $H_{11}$ needs to be factorized
before the Schur complement is able to be handled.
Note that \texttt{HODLR\_LU} has to take a $(\mathcal{T}_{\ell}, p)$-HODLR matrix as input, 
so a rank-$p$ truncation might be required for the matrix $H_{22}^{(k+1)} - L_{21} U_{12}$ in line~\ref{alg.hmlu.line-rank-truncation}
since $\text{rank}(A + B) \le \text{rank}(A) + \text{rank}(B)$ for all $A \in \mathbb{R}^{m \times p}$ and $B \in \mathbb{R}^{p \times n}$. This truncation step will be taken into account in our error analysis \new{afterwards, with the following assumption, which concerns the truncation error incurred in the rank-$p$ truncation for the \textit{computed} Schur complement matrix.}

\begin{assumption}\label{luassum:truncate}
\new{The rank truncation to the \textit{computed} Schur complement $\widetilde{S}^{(k+1)}_{22} := \mathrm{fl}\big(H_{22}^{(k+1)} - L_{21} U_{12}\big)$ in line~\ref{alg.hmlu.line-rank-truncation} of Algorithm~\ref{alg:hmlu} such that the truncated Schur complement $\widehat{S}^{(k+1)}_{22}$ is a $(\mathcal{T}_{\ell}, p)$-HODLR matrix
satisfies, for $0\le k\le \ell -1$,
\begin{equation}\label{eq:LUassmp-truncate}
   \|\widehat{S}^{(k+1)}_{22} - \widetilde{S}^{(k+1)}_{22}\|_F 
   \le \varepsilon \|\widetilde{S}^{(k+1)}_{22}\|_F.
\end{equation}
}
\end{assumption}

\begin{algorithm}[!t]
\caption{HODLR\_LU($H; \mathcal{T}_{\ell}; k=0$)}
\label{alg:hmlu}
\begin{algorithmic}[1]
\IF{$k=\ell$}
\STATE{$L, U \leftarrow$ Compute LU factorization of  $H^{(\ell)}$}\label{alg.hmlu.line-bottom-level-LU}
\RETURN $L, U$
\ELSIF{$k < \ell$}
\STATE{Partition $H^{(k)}$ into $\begin{bmatrix}
    H^{(k+1)}_{11} &  H^{(k+1)}_{12} \\
    H^{(k+1)}_{21} &  H^{(k+1)}_{22}
\end{bmatrix}$   }
\STATE{$L_{11}, U_{11} \leftarrow \texttt{HODLR\_LU}(H_{11}^{(k+1)}, \mathcal{T}_{\ell}, k+1)$}
\STATE{$U_{12} \leftarrow \text{Solve triangular system } L_{11}U_{12} =  H^{(k+1)}_{12}$ \textcolor{blue}{\small [Compute in precision $u$]} } \label{alg.hmlu.tri.solve.1}
\STATE{$L_{21} \leftarrow \text{Solve triangular system } L_{21} U_{11} = H^{(k+1)}_{21}$ \textcolor{blue}{\small [Compute in precision $u$]}} \label{alg.hmlu.tri.solve.2}
\STATE{$H^{\varepsilon}_{22} \leftarrow H_{22}^{(k+1)}  -L_{21} U_{12}$ \textcolor{blue}{\small [Compute in precision $u$]}}\label{alg.hmlu.line-rank-truncation}
\STATE{$L_{22}, U_{22} \leftarrow \texttt{HODLR\_LU}(H^{\varepsilon}_{22}, k+1)$}
\ENDIF
\STATE{$L \leftarrow \begin{bmatrix}
    L_{11} & \\
    L_{21} &  L_{22} 
\end{bmatrix}$, $U \leftarrow \begin{bmatrix}
    U_{11} & U_{12} \\
    &  U_{22} 
\end{bmatrix}$}
\RETURN $L, U$
\end{algorithmic}
\end{algorithm}

\new{To investigate the backward error of the LU factorization of the mixed-precision HODLR matrix $\widehat{H}$, we also introduce the following two assumptions for the analysis of the triangular solvers in lines~\ref{alg.hmlu.tri.solve.1}--\ref{alg.hmlu.tri.solve.2} in Algorithm~\ref{alg:hmlu}, which are based on the results from~\cite[sect. 3.5]{high:ASNA2} and~\cite[Thm. 8.5]{high:ASNA2}.}

\begin{assumption}\label{luassum:1}
The computed approximation $\widehat{C}$ to $C= A_H B_H$,  where $A_H, B_H$ are HODLR matrices for $A \in \mathbb{R}^{n \times n}$ and  $B \in \mathbb{R}^{n \times n}$, respectively, satisfies
\begin{equation}\label{eq:LUassmp-matprod2}
    \widehat{C} = A_H B_H + \Delta C, \qquad \|\Delta C\|_F \le h(n)u \|A\|_F \|B\|_F,
\end{equation}
\end{assumption}
where $h(n)$ is a constant related to $n$.

\begin{assumption}\label{luassum:2}
The computed solution $\widehat{X}$ to the triangular systems $T_H X = B$, where $T_H$ is a HODLR format of $T \in \mathbb{R}^{n \times n}$ and $B \in \mathbb{R}^{n \times n}$, satisfies
\begin{equation}\label{eq:LUassmp-trisolve2}
    T_H \widehat{X} = B + \Delta B, \qquad \|\Delta B\|_F \le f(n) u\|T\|_F \|\widehat{X}\|_F, 
\end{equation}
where $f(n)$ is a constant related to $n$.
\end{assumption}

Note that the constants $h(n)$ and $f(n)$ in~\eqref{eq:LUassmp-matprod2} and~\eqref{eq:LUassmp-trisolve2} 
satisfy $h(n)u=\gamma_n$ and $f(n)u=\gamma_n$
if all the matrices involved are dense.
These constants are introduced to
characterize the error occurring in the \textit{hierarchically-blocked} matrix multiplication and left- and right triangular divisions, 
in which case the constants $h(n)$ and $f(n)$ are expected to be smaller than $\gamma_n / u$.
With these assumptions made, we are now ready to look at the backward error in the LU factorization of the HODLR matrices at level $\ell-1$. Note that the following analysis also takes into account the approximation error in $\widehat{H}$, so the backward error is given with respect to the original $H$ matrix.

\begin{lemma}\label{thm:luhodlr-ell-1}
Let $\widehat{H}$ be the mixed-precision HODLR representation computed via Algorithm~\ref{alg:mp-hodlr-adap}. Under Assumptions~\ref{luassum:1}--\ref{luassum:2}, the LU factors of all HODLR block matrices at level $\ell - 1$ (denoted by $H^{(\ell - 1)}_{ii}, i=1, \ldots, 2^{\ell - 1}$) computed using Algorithm~\ref{alg:hmlu} satisfy
$\widehat{L}_{ii}^{(\ell - 1)} \widehat{U}_{ii}^{(\ell - 1)} = H_{ii}^{(\ell - 1)} + \Delta H_{ii}^{(\ell - 1)}$, where
\begin{equation*}
\|\Delta H_{i i}^{(\ell - 1)}\|_F \le 
u \| H^{(\ell - 1)}_{ii}\|_F  + 
c_1(n, \ell, u, \varepsilon) \|\widehat{L}^{(\ell - 1)}_{ii}\|_F
\|\widehat{U}^{(\ell - 1)}_{ii}\|_F + O(u^2),
\end{equation*} where 
$c_1(n, \ell, u, \varepsilon)= 2\gamma_{n/2^{\ell}}+
(1 + h(n/2^{\ell}))u
+ 2\rho(n,\ell,u,\varepsilon)$,
and where 
$\rho(n,\ell,u,\varepsilon) = (2^{1-\ell/2} + 1)\varepsilon  + f(n/2^\ell) u$.
\end{lemma}

\begin{proof}
From \eqref{eq:LUassmp-trisolve2} we have,
for $k=1\colon\ell$,
\begin{align*}
  T \widehat{X} &= \widehat{H}_{ij}^{(k)} + \Delta_F^{(k)}, \qquad \|\Delta_F^{(k)}\|_F \le  f(n/2^k) u \|T\|_F \|\widehat{X}\|_F. \\
  &\equiv H_{ij}^{(k)} + \Delta_L^{(k)} + \Delta_F^{(k)} \equiv 
  H_{ij}^{(k)} + \Delta H_{ij}^{(k)},
\end{align*}
where, by using~\eqref{eq:error-adap-prec-offdiagbloc} and ignoring higher order terms in $\varepsilon$ and $u$, 
we have (note that the blocks on level $k$ have size $n/2^{k}$) 
\begin{align}\label{eq:LUassmp-trisolve_lowrank}
    \normF{\Delta H_{ij}^{(k)}} &\le \normF{\Delta_L^{(k)}} + 
    \normF{\Delta_F^{(k)}} 
    \lesssim (2^{1-k/2} + 1) \varepsilon\norm{H_{ij}^{(k)}}_F + 
     f(n/2^k) u\|T\|_F \|\widehat{X}\|_F \nonumber\\
    &\lesssim 
    \left((2^{1-k/2} + 1)\varepsilon  + f(n/2^k) u\right) 
    \|T\|_F \|\widehat{X}\|_F
    =: \rho(n, k, u, \varepsilon)\|T\|_F \|\widehat{X}\|_F.
\end{align}
Similarly, \eqref{eq:LUassmp-trisolve2} also implies, for $k=1\colon\ell$,
\begin{equation}\label{eq:LUassmp-trisolve_lowrank2}
    \widehat{X}T  = H_{ij}^{(k)} + \Delta H_{ij}^{(k)},
    \qquad \normF{\Delta H_{ij}^{(k)}}\lesssim
    \rho(n, k, u, \varepsilon)\|T\|_F \|\widehat{X}\|_F.
\end{equation}

Regarding the Schur complement term 
$S^{(\ell)}_{2i} := H_{2i,2i}^{(\ell)} 
- \widehat{L}^{(\ell)}_{2i,2i-1} 
\widehat{U}^{(\ell)}_{2i-1,2i} =:
H_{2i,2i}^{(\ell)} - Z_{2i}^{\ell}$, 
we have
\begin{align}\label{eq:LU-level-ell-1-SchurComple}
\widehat{S}^{(\ell)}_{2i} 
&=  \mathrm{fl}(
H_{2i,2i}^{(\ell)} 
- \widehat{L}^{(\ell)}_{2i,2i-1} 
\widehat{U}^{(\ell)}_{2i-1,2i})
=: H_{2i,2i}^{(\ell)} -  Z_{2i}^{(\ell)} 
+ \Delta Z_{2i}^{(\ell)} + \Delta_{S_{2i}}^{(\ell)}, 
\end{align}
where $\Delta Z_{2i}^{(\ell)}$ denotes the error occurring in forming the matrix product $\widehat{L}^{(\ell)}_{2i,2i-1} \widehat{U}^{(\ell)}_{2i-1,2i}$ and $\Delta_{S_{2i}}^{(\ell)}$ denotes the error resulting from the subsequent matrix summations. Using the assumption~\eqref{eq:LUassmp-matprod2} and ignoring higher order terms, we have 
\begin{equation}\label{eq:LU-level-ell-1-SchurComple-bnd-prod}
    \normF{\Delta Z_{2i}^{(\ell)}} \le h(n/2^\ell)u \normF{\widehat{L}^{(\ell)}_{2i,2i-1}} \normF{\widehat{U}^{(\ell)}_{2i-1,2i}}
\end{equation} 
and
\begin{equation}\label{eq:LU-level-ell-1-SchurComple-bnd-sum}
    \normF{\Delta_{S_{2i}}^{(\ell)}} \le 
    u (\normF{H_{2i,2i}^{(\ell)}} + \normF{Z_{2i}^{(\ell)}})
    \le u (\normF{H_{2i,2i}^{(\ell)}} + \normF{\widehat{L}^{(\ell)}_{2i,2i-1}} \normF{\widehat{U}^{(\ell)}_{2i-1,2i}}).
\end{equation}
On the other hand, from~\cite[Thm.~9.3]{high:ASNA2}, the computed LU factorization of the computed Schur complement $\widehat{S}^{(\ell)}_{2i}$ satisfies
\begin{equation}\label{eq:LU-level-ell-1-SchurComple-bnd-lu}
\widehat{L}^{(\ell)}_{2i,2i}\widehat{U}^{(\ell)}_{2i,2i} =
    \widehat{S}^{(\ell)}_{2i} + \Delta F^{(\ell)}_{2i}, \qquad \|\Delta F^{(\ell)}_{2i}\|_F \le 
    \gamma_{n/2^{\ell}} 
    \|\widehat{L}^{(\ell)}_{2i,2i}\|_F
    \|\widehat{U}^{(\ell)}_{2i,2i}\|_F.
\end{equation}
Combining~\eqref{eq:LU-level-ell-1-SchurComple}--\eqref{eq:LU-level-ell-1-SchurComple-bnd-lu}, we arrive at
\begin{align}\label{eq:LU-level-ell-1-SchurComple-bnd}
\widehat{L}^{(\ell)}_{2i,2i}\widehat{U}^{(\ell)}_{2i,2i}+
\widehat{L}^{(\ell)}_{2i,2i-1} 
\widehat{U}^{(\ell)}_{2i-1,2i} &=
    H_{2i,2i}^{(\ell)} 
+ \Delta Z_{2i}^{(\ell)} + \Delta_{S_{2i}}^{(\ell)} + \Delta F^{(\ell)}_{2i} =: H_{2i,2i}^{(\ell)} 
+ \Delta H_{2i,2i}^{(\ell)},
\end{align}
where
\begin{multline*}
\normF{\Delta H_{2i,2i}^{(\ell)}}  \le
\normF{\Delta Z_{2i}^{(\ell)}} + \normF{\Delta_{S_{2i}}^{(\ell)}} + 
\normF{\Delta F^{(\ell)}_{2i}} \\
\qquad \quad 
\le u\|H_{2i,2i}^{(\ell)}\|_F + (1 + h(n/2^{\ell}) )u   \|\widehat{L}^{(\ell)}_{2i,2i-1}\|_F
\|\widehat{U}^{(\ell)}_{2i-1,2i}\|_F 
+ \gamma_{n/2^{\ell}} \|\widehat{L}^{(\ell)}_{2i,2i}\|_F
\|\widehat{U}^{(\ell)}_{2i,2i}\|_F.
\end{multline*}
   
Note the matrix $\widehat{H}_{ii}^{(\ell-1)}$, $i=1,\ldots, 2^{\ell-1}$, on level $\ell-1$ is partitioned the same way as in~\eqref{eq:hodlr_block}.
Using~\cite[Thm. 9.3]{high:ASNA2} and combining~\eqref{eq:LUassmp-trisolve_lowrank}--\eqref{eq:LUassmp-trisolve_lowrank2} and~\eqref{eq:LU-level-ell-1-SchurComple-bnd} on the block LU factorization
\begin{equation*}
H^{(\ell-1)}_{ii} + \Delta H^{(\ell-1)}_{ii} =
\widehat{L}^{(\ell-1)}_{ii} \widehat{U}^{(\ell-1)}_{ii} \equiv 
\begin{bmatrix}
    \widehat{L}^{(\ell)}_{2i-1,2i-1}  &  \\  \widehat{L}^{(\ell)}_{2i,2i-1}  & \widehat{L}^{(\ell)}_{2i,2i}
\end{bmatrix}
\begin{bmatrix}
    \widehat{U}^{(\ell)}_{2i-1,2i-1}  & \widehat{U}^{(\ell)}_{2i-1,2i} \\  
    & \widehat{U}^{(\ell)}_{2i,2i}
\end{bmatrix},
\end{equation*}

we arrive at 
\begin{equation*}
\begin{aligned}
\|\Delta H^{(\ell - 1)}_{ii}\|_F \le& 
\|\Delta H^{(\ell)}_{2i-1,2i-1}\|_F + \|\Delta \widehat{H}^{(\ell)}_{2i-1,2i}\|_F + \|\Delta \widehat{H}^{(\ell)}_{2i,2i-1}\|_F + \|\Delta H^{(\ell)}_{2i,2i}\|_F \\
\le& u\|H_{2i,2i}^{(\ell)}\|_F + \gamma_{n/2^{\ell}} \|\widehat{L}^{(\ell)}_{2i-1,2i-1}\|_F\|\widehat{U}^{(\ell)}_{2i-1,2i-1}\|_F \\
&+ \rho(n,\ell,u,\varepsilon)
\big(\|\widehat{L}^{(\ell)}_{2i-1,2i-1}\|_F\|\widehat{U}^{(\ell)}_{2i-1,2i}\|_F 
+ \|\widehat{L}^{(\ell)}_{2i,2i-1}\|_F
\|\widehat{U}^{(\ell)}_{2i-1,2i-1}\|_F\big)  \\
&+ (1 + h(n/2^{\ell}))u  \|\widehat{L}^{(\ell)}_{2i,2i-1}\|_F
\|\widehat{U}^{(\ell)}_{2i-1,2i}\|_F 
+ \gamma_{n/2^{\ell}} \|\widehat{L}^{(\ell)}_{2i,2i}\|_F
\|\widehat{U}^{(\ell)}_{2i,2i}\|_F.
\end{aligned}
\end{equation*}
This bound can be weakened and expressed together as 

\begin{equation*}
 \|\Delta H^{(\ell - 1)}_{ii}\|_F   \le 
 u \| H^{(\ell - 1)}_{ii}\|_F +
 c_1(n, \ell, u, \varepsilon) 
\|\widehat{L}^{(\ell-1)}_{ii}\|_F
\|\widehat{U}^{(\ell-1)}_{ii}\|_F,  
\end{equation*}
where $c_1(n, \ell, u, \varepsilon)= 2\gamma_{n/2^{\ell}}+
(1 + h(n/2^{\ell}))u
+ 2\rho(n,\ell,u,\varepsilon)$.
\end{proof}

Now we develop error bounds for the LU factorization of HODLR matrices at arbitrary levels. This will generalize the result of~Lemma~\ref{thm:luhodlr-ell-1} for the penultimate level and necessarily weaken the bound. The major difference in the generalization is that block LU instead of dense LU will be invoked, as presented in Algorithm~\ref{alg:hmlu}.
In general, the quantities $U_{12}$ and $L_{21}$ returned by  Algorithm~\ref{alg:hmlu} can be dense or in HODLR format, and we consider the latter by using Assumptions~\ref{luassum:1}--\ref{luassum:2}.

\begin{theorem}[Global error in LU factorization of mixed-precision HODLR matrices]\label{thm:luhodlr}
Let $\widehat{H}$ be the mixed-precision HODLR representation computed via Algorithm~\ref{alg:mp-hodlr-adap}.
In computing the LU decomposition of $\widehat{H}$,
\new{under Assumptions~\ref{luassum:truncate}--\ref{luassum:2},} the LU factors of $(\mathcal{T}_{\ell}, p)$-HODLR matrices of size $n/2^{\ell'} 
 (0 \le \ell' \le \ell-1)$ satisfy 
\begin{equation}\label{eq:lu_hodlr_gen_errbnd}
\begin{aligned}
\widehat{L}^{(\ell')}_{ii} \widehat{U}^{(\ell')}_{ii} &=  H^{(\ell')}_{ii} + \Delta H^{(\ell')}_{ii}, 
\end{aligned}
\end{equation} 
where 
\begin{equation*}
\|\Delta H_{i i}^{(\ell')}\|_F \le
(2^{\ell - \ell'}-1)(2\varepsilon + u)\| H^{(\ell')}_{ii}\|_F  + 
c_2(n, \ell', u, \varepsilon) \|\widehat{L}^{(\ell')}_{ii}\|_F
\|\widehat{U}^{(\ell')}_{ii}\|_F + O(u^2),
\end{equation*} 
and where 
$c_2(n, \ell', u, \varepsilon)=  (2^{\ell - \ell'}-1)(\gamma_{n/2^{\ell'}}+
2\rho(n,\ell',u,\varepsilon)+\varepsilon+ h(n/2^{\ell'}) u)$.
In particular, the LU factorization of the HODLR matrix $\widehat{H}$ (level $0$) satisfies $\widehat{L} \widehat{U} =  H + \Delta H$, where
\begin{equation*}
\|\Delta H \|_F  \le (2^{\ell}-1) (2\varepsilon+u) \| H \|_F  + 
(2^{\ell}-1)(7\varepsilon + \gamma_{n}+2f(n)u +h(n)u)\|\widehat{L}\|_F\|\widehat{U}\|_F+ O(u^2).
\end{equation*}
\end{theorem}

\begin{proof}
The proof proceeds by induction on the backward error in the LU decomposition of diagonal blocks in different levels. 

For the base case of level $\ell'=\ell-1$, from
Lemma~\ref{thm:luhodlr-ell-1} the 
bound~\eqref{eq:lu_hodlr_gen_errbnd} holds since 
$c_1(n, \ell, u, \varepsilon)= 2\gamma_{n/2^{\ell}} +
 2\rho(n,\ell,u,\varepsilon) + u + h(n/2^{\ell}) u
\le \gamma_{n/2^{\ell-1}}+
2\rho(n,\ell-1,u,\varepsilon)+\varepsilon+ h(n/2^{\ell-1}) u =
c_2(n, \ell-1, u, \varepsilon)$.

For the inductive step, assume~\eqref{eq:lu_hodlr_gen_errbnd} is satisfied for HODLR blocks in level $\ell'=k+1$, $0\le k\le\ell-2$. 
A level-$k$ diagonal block of $\widehat{H}$, $0\le k\le\ell-2$,
is partitioned the same way as in~\eqref{eq:hodlr_block}, and its block LU factorization satisfies
\begin{equation}\label{eq:lu-levelk}
\widehat{H}^{(k)}_{ii} =
\widehat{L}^{(k)}_{ii} \widehat{U}^{(k)}_{ii} \equiv 
\begin{bmatrix}
    \widehat{L}^{(k+1)}_{2i-1,2i-1}  &  \\  \widehat{L}^{(k+1)}_{2i,2i-1}  & \widehat{L}^{(k+1)}_{2i,2i}
\end{bmatrix}
\begin{bmatrix}
    \widehat{U}^{(k+1)}_{2i-1,2i-1}  & \widehat{U}^{(k+1)}_{2i-1,2i} \\  
    & \widehat{U}^{(k+1)}_{2i,2i}
\end{bmatrix}.
\end{equation}
From our inductive hypothesis, the LU factorization of $\widehat{H}^{(k+1)}_{2i-1,2i-1}$ satisfies
\begin{equation}\label{eq:lua11p1}
\widehat{L}^{(k+1)}_{2i-1,2i-1} \widehat{U}^{(k+1)}_{2i-1,2i-1} =  H^{(k+1)}_{2i-1,2i-1} + \Delta H^{(k+1)}_{2i-1,2i-1},  
\end{equation}
where 
\begin{multline*}
\|\Delta H_{2i-1,2i-1}^{(k+1)}\|_F  \lesssim 
(2^{\ell - k - 1}-1)(2\varepsilon+u)\| H^{(k+1)}_{2i-1,2i-1}\|_F  \\ + c_2(n,k+1, u, \varepsilon) \|\widehat{L}^{(k+1)}_{2i-1,2i-1}\|_F
\|\widehat{U}^{(k+1)}_{2i-1,2i-1}\|_F.
\end{multline*}

Applying~\eqref{eq:LUassmp-trisolve_lowrank}--\eqref{eq:LUassmp-trisolve_lowrank2} to the triangular solves with
$\widehat{H}^{(k+1)}_{2i-1,2i}$ and $\widehat{H}^{(k+1)}_{2i,2i-1}$ as the right-hand sides, respectively, we get
\begin{equation}\label{eq:lu-offdiag-trianglesove-err1}
\widehat{L}^{(k+1)}_{2i-1,2i-1} \widehat{U}^{(k+1)}_{2i-1,2i} 
= H^{(k+1)}_{2i-i,2i} + \Delta H^{(k+1)}_{2i-1,2i}, 
\end{equation}
where $\|\Delta H^{(k+1)}_{2i-1,2i}\|_F \le \rho(n, k+1, u, \varepsilon) \|\widehat{L}_{2i-1,2i-1}^{(k+1)}\|_F\|\widehat{U}^{(k+1)}_{2i-1,2i}\|_F$, and
\begin{equation}\label{eq:lu-offdiag-trianglesove-err2}
\widehat{L}_{2i,2i-1}^{(k+1)} \widehat{U}^{(k+1)}_{2i-1,2i-1}
= H^{(k+1)}_{2i,2i-1} + \Delta H^{(k+1)}_{2i,2i-1},
\end{equation}
where $\|\Delta H^{(k+1)}_{2i,2i-1}\|_F \le \rho(n, k+1, u, \varepsilon) \|\widehat{L}_{2i,2i-1}^{(k+1)}\|_F\|\widehat{U}^{(k+1)}_{2i-1,2i-1}\|_F$.

For the Schur complement term $S^{(k+1)}_{2i} := \widehat{H}_{2i,2i}^{(k+1)} - \widehat{L}^{(k+1)}_{2i,2i-1} \widehat{U}^{(k+1)}_{2i-1,2i}=: \widehat{H}_{2i,2i}^{(k+1)} - Z^{(k+1)}_{2i}$, from Assumption~\ref{luassum:1} we have
\begin{equation}\label{eq:temp_zl-1}
    \widehat{Z}^{(k+1)}_{2i} =  \widehat{L}^{(k+1)}_{2i,2i-1}\widehat{U}^{(k+1)}_{2i-1,2i} - \Delta Z^{(k+1)}_{2i}, \  \|\Delta Z^{(k+1)}_{2i}\|_F \le h(n/2^{k+1})u\|\widehat{L}^{(k+1)}_{2i,2i-1}\|_F\|\widehat{U}^{(k+1)}_{2i-1,2i}\|_F,
\end{equation}
and the computed Schur complement $\widetilde{S}^{(k+1)}_{2i}$ satisfies
\begin{equation}\label{eq:schur_fp2}
    \widetilde{S}^{(k+1)}_{2i} = \widehat{H}_{2i,2i}^{(k+1)} - \widehat{Z}^{(k+1)}_{2i} + F^{(k+1)}_{2i}, \   \|F^{(k+1)}_{2i}\|_F \le u(\|\widehat{H}_{2i,2i}^{(k+1)}\|_F + \|\widehat{Z}^{(k+1)}_{2i}\|_F). 
\end{equation}
\new{From Assumption~\ref{luassum:truncate},} the truncation error $\Delta \tau^{(k+1)}_{2i}:= \widehat{S}^{(k+1)}_{2i} - \widetilde{S}^{(k+1)}_{2i}$ in this process satisfies 
$\|\Delta\tau^{(k+1)}_{2i}\|_F \le \varepsilon \|\widetilde{S}^{(k+1)}_{2i}\|_F
\approx \varepsilon \|\widehat{S}^{(k+1)}_{2i}\|_F$.
Combining this truncation error bound with~\eqref{eq:temp_zl-1}--\eqref{eq:schur_fp2}, we have, with higher order terms in
$u$ and $\varepsilon$ ignored,
\begin{equation}\label{eq:schur_fp3}
\begin{aligned}
    \widehat{S}^{(k+1)}_{2i} &=\widetilde{S}^{(k+1)}_{2i} + \Delta\tau^{(k+1)}_{2i}
    =: H_{2i,2i}^{(k+1)} - \widehat{L}^{(k+1)}_{2i,2i-1}
    \widehat{U}^{(k+1)}_{2i-1,2i} + \Delta S^{(k+1)}_{2i}, 
\end{aligned}
\end{equation}
where $\Delta S^{(k+1)}_{2i}=\Delta \widehat{H}_{2i,2i}^{(k+1)} + \Delta Z^{(k+1)}_{2i} + F^{(k+1)}_{2i}+ \Delta\tau^{(k+1)}_{2i}$ (recalling from Lemma~\ref{lemma:approx-diag} that $\Delta \widehat{H}_{2i,2i}^{(k+1)}=\widehat{H}_{2i,2i}^{(k+1)} - H_{2i,2i}^{(k+1)}$ denotes the representation error) and 
\begin{equation*}
\begin{aligned}
    \|\Delta S^{(k+1)}_{2i}\|_F 
    \lesssim& \varepsilon\|H_{2i,2i}^{(k+1)}\|_F + h(n/2^{k+1})u  \|\widehat{L}^{(k+1)}_{2i,2i-1}\|_F
    \|\widehat{U}^{(k+1)}_{2i-1,2i}\|_F \\
    & + u(\|H_{2i,2i}^{(k+1)}\|_F  + \|\widehat{Z}_{2i}^{(k+1)}\|_F)+ \varepsilon \|{H}_{2i,2i}^{(k+1)} - \widehat{L}^{(k+1)}_{2i,2i-1}\widehat{U}^{(k+1)}_{2i-1,2i} \|_F \\
    \lesssim& (2\varepsilon+u)\|H_{2i,2i}^{(k+1)}\|_F + \left(h(n/2^{k+1})u  + \varepsilon \right)\|\widehat{L}^{(k+1)}_{2i,2i-1}\|_F
    \|\widehat{U}^{(k+1)}_{2i-1,2i}\|_F.
\end{aligned}
\end{equation*}

Using the inductive hypothesis again, the computation of the LU factorization of the truncated Schur complement $\widehat{S}^{(k+1)}_{2i}$ satisfies
\begin{equation}\label{eq:luA22p}
    \widehat{L}^{(k+1)}_{2i,2i} \widehat{U}^{(k+1)}_{2i,2i} = \widehat{S}^{(k+1)}_{2i} + \Delta \widehat{S}^{(k+1)}_{2i},
\end{equation}
where 
\begin{equation*}
\|\Delta \widehat{S}^{(k+1)}_{2i}\|_F \lesssim 
(2^{\ell - k - 1}-1)(2\varepsilon+u)\| H^{(k+1)}_{2i,2i}\|_F  
+ c_2(n, k+1, u, \varepsilon) \|\widehat{L}^{(k+1)}_{2i,2i}\|_F
\|\widehat{U}^{(k+1)}_{2i,2i}\|_F.
\end{equation*}
Using \eqref{eq:schur_fp3} and \eqref{eq:luA22p}, we have
\begin{equation}\label{eq:luA22p2}
\begin{aligned}
\widehat{L}^{(k+1)}_{2i,2i} \widehat{U}^{(k+1)}_{2i,2i} + \widehat{L}^{(k+1)}_{2i,2i-1}\widehat{U}^{(k+1)}_{2i-1,2i} 
&= H_{2i,2i}^{(k+1)} + \Delta S^{(k+1)}_{2i} + \Delta \widehat{S}^{(k+1)}_{2i} \\
&=: H_{2i,2i}^{(k+1)} + \Delta H_{2i,2i}^{(k+1)},
\end{aligned}
\end{equation}
where 
\begin{equation*}
\begin{aligned}
\|\Delta H_{22}^{(k+1)}\|_F& \le \|\Delta S^{(k+1)}_{2i}\|_F +
\|\Delta \widehat{S}^{(k+1)}_{2i}\|_F \\  
\lesssim& 2^{\ell - k - 1}(2\varepsilon+u)\|H_{2i,2i}^{(k+1)}\|_F + \left(h(n/2^{k+1})u  + \varepsilon \right)\|\widehat{L}^{(k+1)}_{2i,2i-1}\|_F
    \|\widehat{U}^{(k+1)}_{2i-1,2i}\|_F\\
&+ c_2(n, k+1, u, \varepsilon) \|\widehat{L}^{(k+1)}_{2i,2i}\|_F
\|\widehat{U}^{(k+1)}_{2i,2i}\|_F.
\end{aligned}
\end{equation*}
Combining~\eqref{eq:lua11p1}--\eqref{eq:lu-offdiag-trianglesove-err2} and \eqref{eq:luA22p2} for the block LU factorization~\eqref{eq:lu-levelk} at level $k$, we arrive at
$\widehat{L}^{(k)}_{ii} \widehat{U}^{(k)}_{ii} = H^{(k)}_{ii} + \Delta H^{(k)}_{ii}$, 
where,
similar to the final steps of Lemma~\ref{thm:luhodlr-ell-1},
we have
\begin{equation*}
\begin{aligned}
\|\Delta H^{(k)}_{ii}\|_F  \le& (2^{\ell-k}-1)(2\varepsilon+u)\| H^{(k)}_{ii}\|_F  \\
+& \left(2c_2(n,k+1, u, \varepsilon) + 2\rho(n, k+1,u,\varepsilon) +
h(n/2^{k+1})u  + \varepsilon \right)
\|\widehat{L}^{(k)}_{ii}\|_F
\|\widehat{U}^{(k)}_{ii}\|_F,
\end{aligned}
\end{equation*}
where $2c_2(n,k+1, u, \varepsilon) =(2^{\ell - k}-2)(\gamma_{n/2^{k+1}}+
2\rho(n,k+1,u,\varepsilon)+\varepsilon+ h(n/2^{k+1}) u)$, and so 
the factor of $\|\widehat{L}^{(k)}_{ii}\|_F
\|\widehat{U}^{(k)}_{ii}\|_F$ in this bound is bounded above by
$(2^{\ell - k}-1)(\gamma_{n/2^{k+1}}+
2\rho(n,k+1,u,\varepsilon)+\varepsilon+ h(n/2^{k+1}) u) \le c_2(n,k, u, \varepsilon)$.
It follows that
\begin{equation*}
  \|\Delta H^{(k)}_{ii}\|_F  \le (2^{\ell-k}-1)(2\varepsilon+u)\| H^{(k)}_{ii}\|_F  + c_2(n,k, u, \varepsilon)\|\widehat{L}^{(k)}_{ii}\|_F
\|\widehat{U}^{(k)}_{ii}\|_F.
\end{equation*}
This proves the bound~\eqref{eq:lu_hodlr_gen_errbnd} for HODLR blocks in level $k$ and so the proof is completed by induction.
\end{proof}

\begin{corollary}\label{thm:luhodlr-epsilon}
Let $\widehat{H}$ be the mixed-precision $\ell$-level HODLR representation computed via Algorithm~\ref{alg:mp-hodlr-adap}. 
If the LU decomposition of $\widehat{H}$ is computed in a working precision $u \lesssim \varepsilon/n$, then
the LU factorization of the HODLR matrix $\widehat{H}$ satisfies \begin{equation}\label{eq:lu_col}
\widehat{L} \widehat{U} =  H + \Delta H, \quad
\|\Delta H \|_F  \lesssim 2(2^{\ell}-1)\varepsilon \| H \|_F  + 
11(2^{\ell}-1)\varepsilon \|\widehat{L}\|_F\|\widehat{U}\|_F.
\end{equation}
\end{corollary}

Corollary~\ref{thm:luhodlr-epsilon} basically says that the precision in which we compute the LU factorization should be based on the parameter $\varepsilon$; if we select a relatively large $\varepsilon$, and thus have a high level of approximation, then we can compute the LU factorization in a low precision without affecting the backward error.
Also, we note that the results remain applicable when the mixed-precision storage of HODLR matrices are not used, in which case 
the results degenerate to the LU factorization of HODLR matrices 
stored in one precision and the bound remains in the same form, albeit with smaller constant factors.

\section{Numerical experiments}\label{sec:nexp}
Our adaptive-precision HODLR matrix construction scheme described in Algorithm~\ref{alg:mp-hodlr-adap} is implemented in MATLAB. The implementation allows any depth and an arbitrary set of precisions to be prescribed. 
All off-diagonal blocks in the HODLR matrix are truncated using the SVD with user-specified parameter $0<\varepsilon<1$. All experiments are run on a 2.30 GHz Intel Core i7-12700H CPU with 32GB RAM, running in a single thread.
All precisions used for computing the adaptive-precision HODLR matrix described in Algorithm~\ref{alg:mp-hodlr-adap} are taken from the set $\{\text{q52, bf16, fp16, fp32, fp64}\}$; see  \tablename~\ref{table:unitroundoff} for details on these precision formats. 
The precisions are simulated by the chop function of \cite{doi:10.1137/19M1251308}.

In the following experiments, we evaluate the reconstruction error, the backward errors of the matrix--vector product and LU factorization, and the storage cost for the HODLR matrix. \new{The assumption of a perfectly balanced tree structure in the HODLR format barely holds in practice. Thus, we remove this constraint in our splitting scheme by setting the size of the left child node as the ceiling of the halved size of the parent node.
Through this practical partitioning scheme, the size of the blocks in the finest level is approximately $\lfloor n/2^{\ell} \rfloor$ or $\lceil n/2^{\ell} \rceil$.}

\new{Our test matrices consist of kernel matrices generated by~\eqref{eq:test-kernals} and the Schur complements (with respect to the $(1,1)$ block, of dimension~$n-\lceil n/2\rceil$, where $n$ is the size of the original matrix) of the the matrices listed in~Table~\ref{tab:testmat}, 
including the matrix P64 from~\cite{abbg22} and a selection from the SuiteSparse collection~\cite{10.1145/2049662.2049663}.
Numerous scientific domains, e.g., integral equations and machine learning, often give rise to such matrices \cite{doi:10.1137/22M1491253, 8425507}. 
HODLR matrices play a crucial role in kernel matrix approximation as well as its applications, which involve, e.g., matrix--vector products (Algorithm~\ref{alg:matvecprod}).
The Schur complement corresponds to the root separator of hierarchical partitionings and its (structured) approximation is crucial in sparse linear solvers~\cite{cdgs10, xcgl10}.}

\begin{table}
\caption{Summary of test matrices from~\cite{abbg22} and the SuiteSparse collection~\cite{10.1145/2049662.2049663}.}\label{tab:testmat}
\centering\footnotesize
\setlength\tabcolsep{5.0pt}
\begin{tabular}{l r r l}
\toprule
Dataset  & Size & Nonzeros & Description  \\ [0.5ex] 
\midrule
1138\_bus& 1,138	&4,054 &Power network problem \\
bcsstk08 & 1,074 &	12,960 & Structural problem\\
cavity18 &	4,562 & 138,040	 & Subsequent computational fluid dynamics problem\\
ex37 & 3,565 & 67,591 & 	Computational fluid dynamics problem\\
LeGresley\_2508 &2,508& 16,727 &	Power network problem \\
P64& 4,096 &  16,777,216 & The discretization of a Poisson equation   \\
psmigr\_1& 3,140 &	543,160 & Economic problem \\
saylr3& 1,000	 & 3,750& 	Computational fluid dynamics problem \\

\bottomrule
\end{tabular}
\label{realdatainfo}
\end{table}

\subsection{Global construction error}

\begin{figure}
\centering
\subfigure[ex37]{
\includegraphics[width=0.4\textwidth]{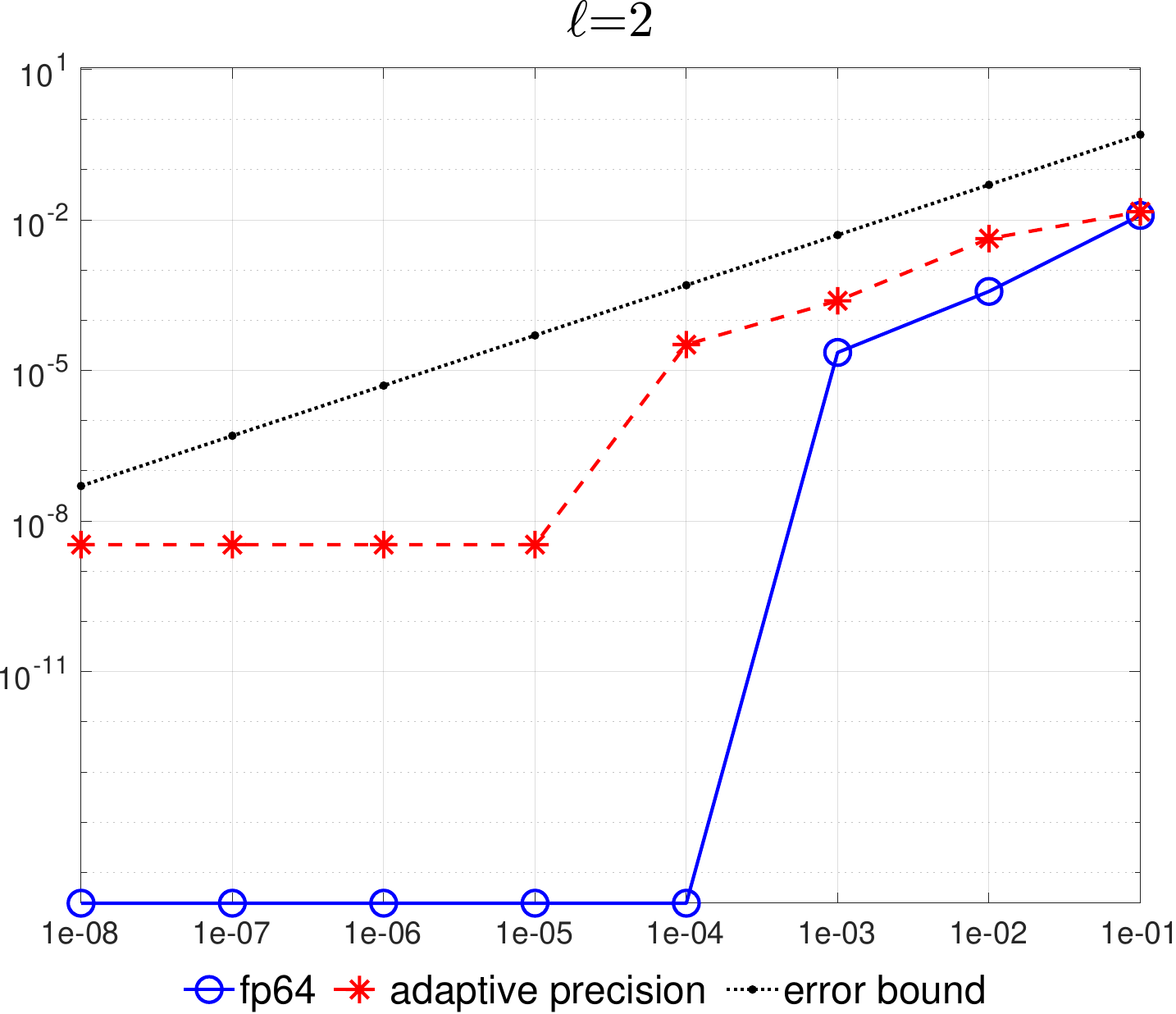}
\includegraphics[width=0.4\textwidth]{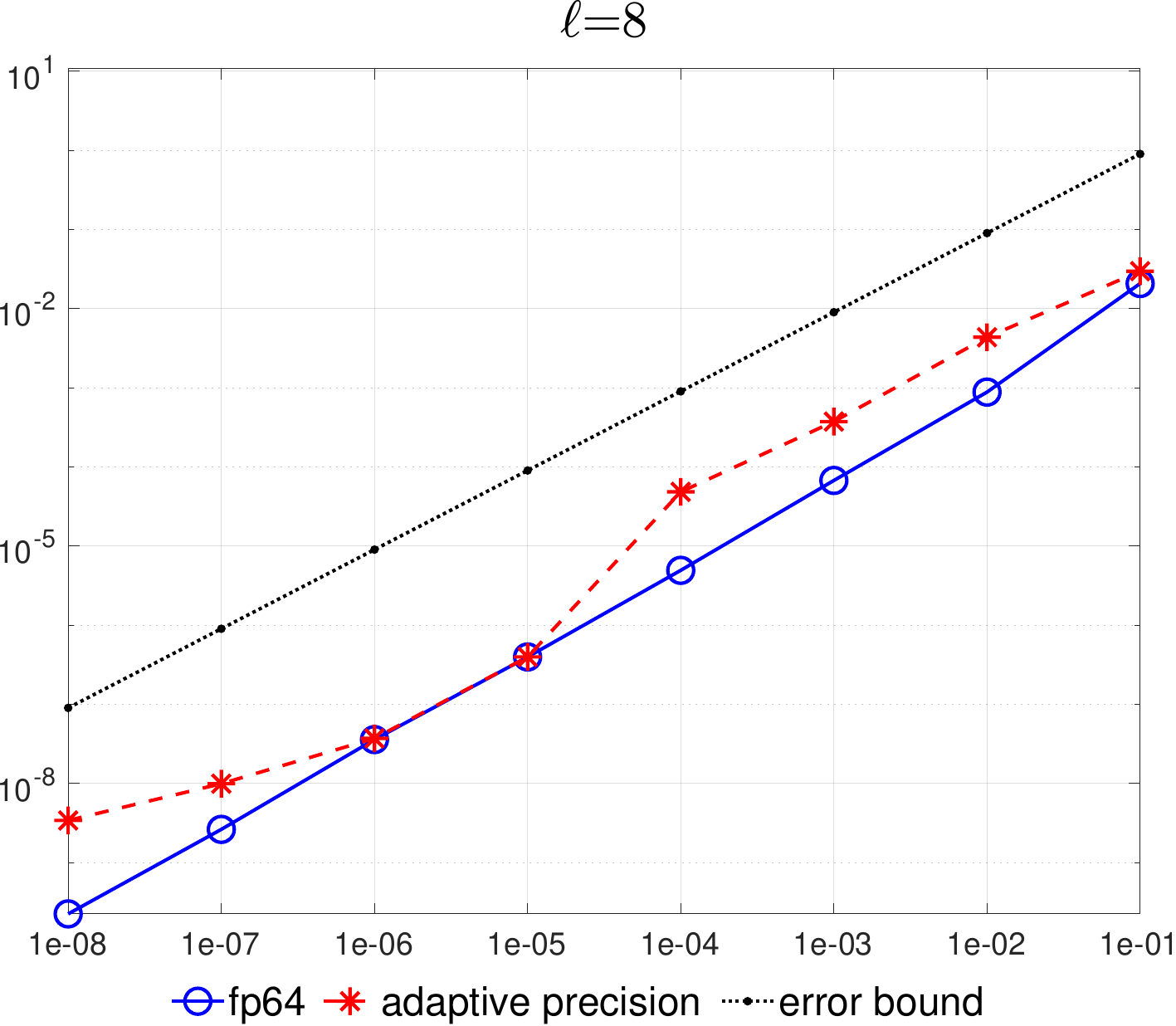}
}\\
\vspace{-10pt}
\subfigure[P64]{
\includegraphics[width=0.4\textwidth]{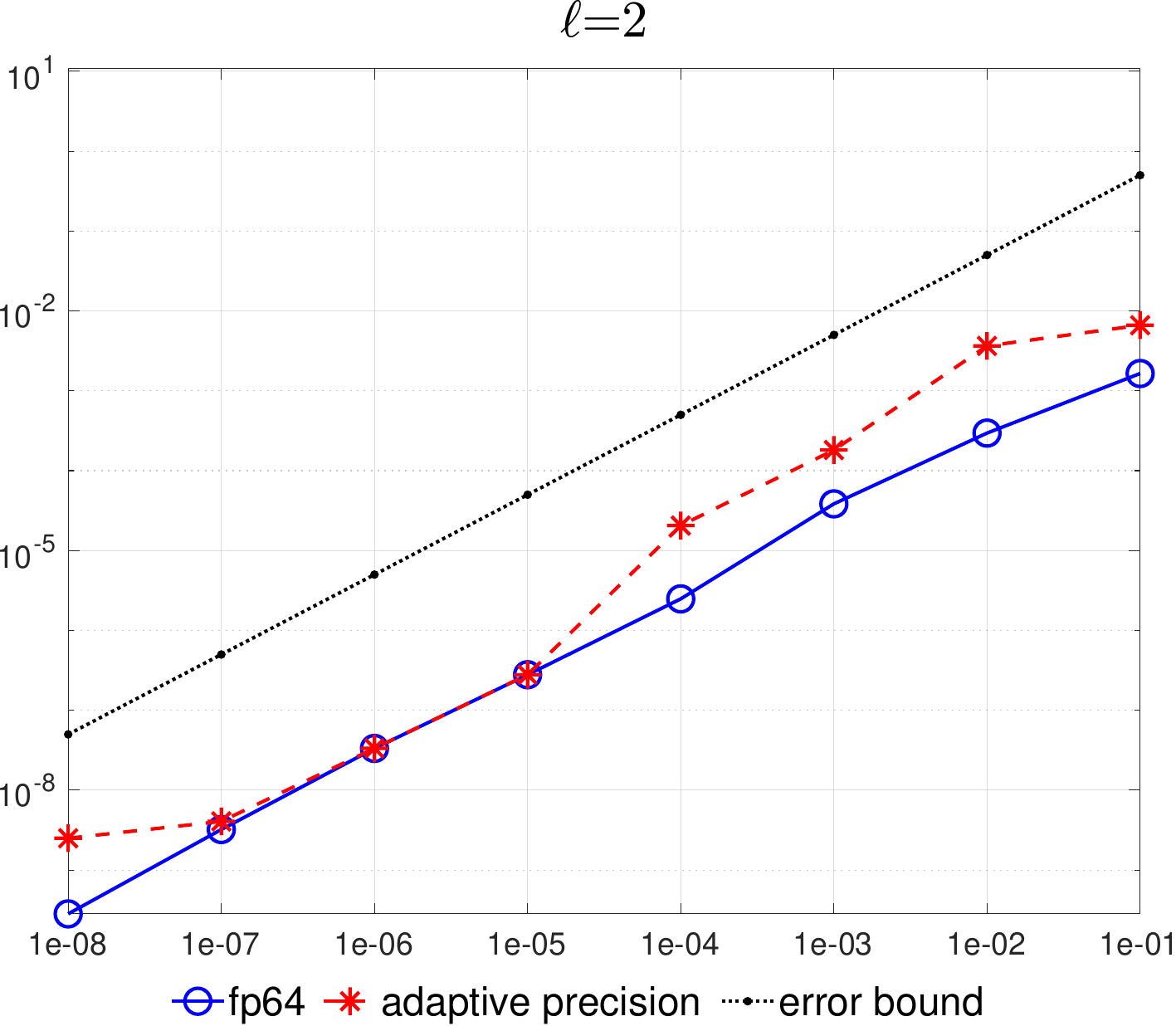}
\includegraphics[width=0.4\textwidth]{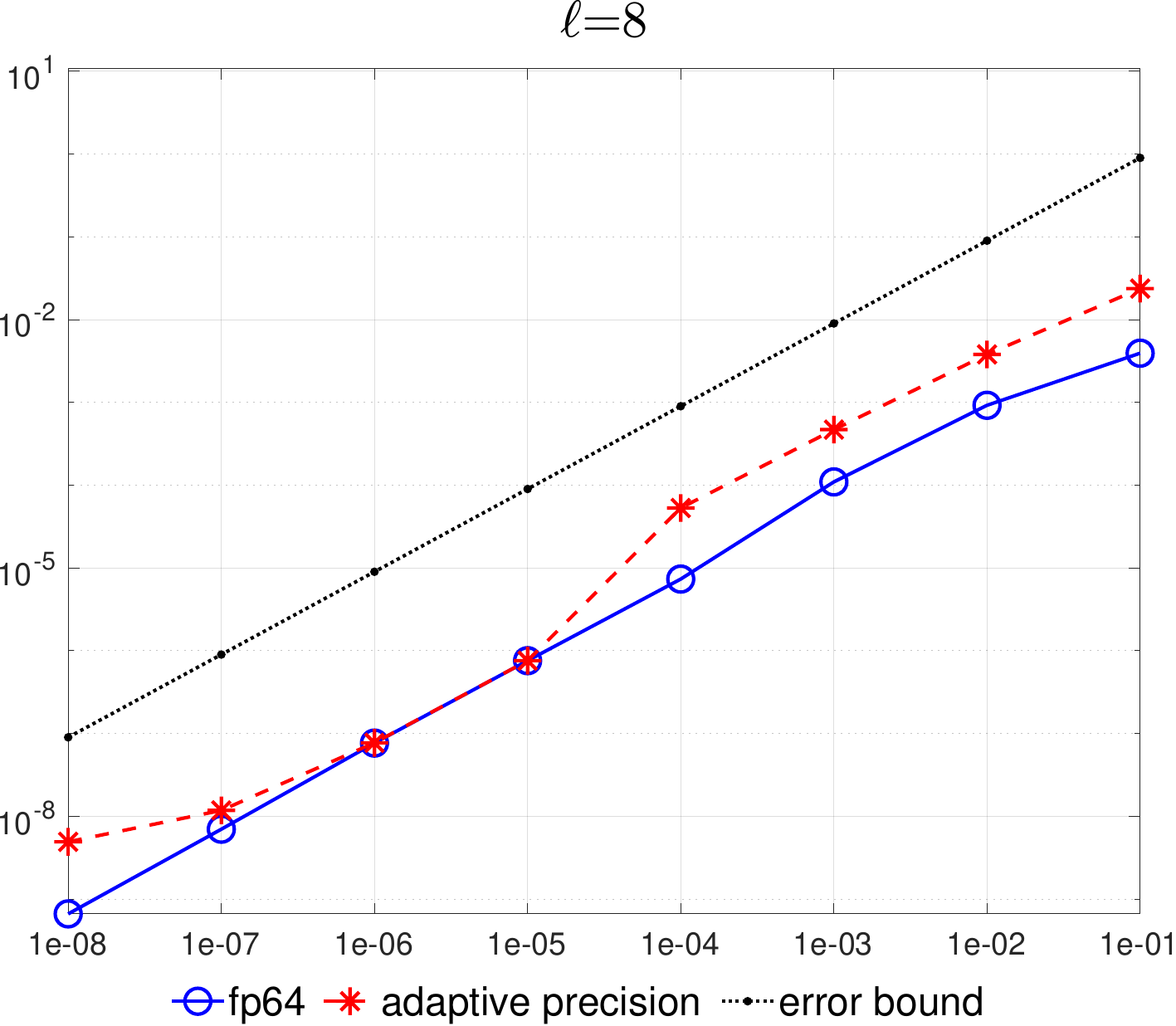}
}\\
\vspace{-10pt}
\subfigure[saylr3]{
\includegraphics[width=0.4\textwidth]{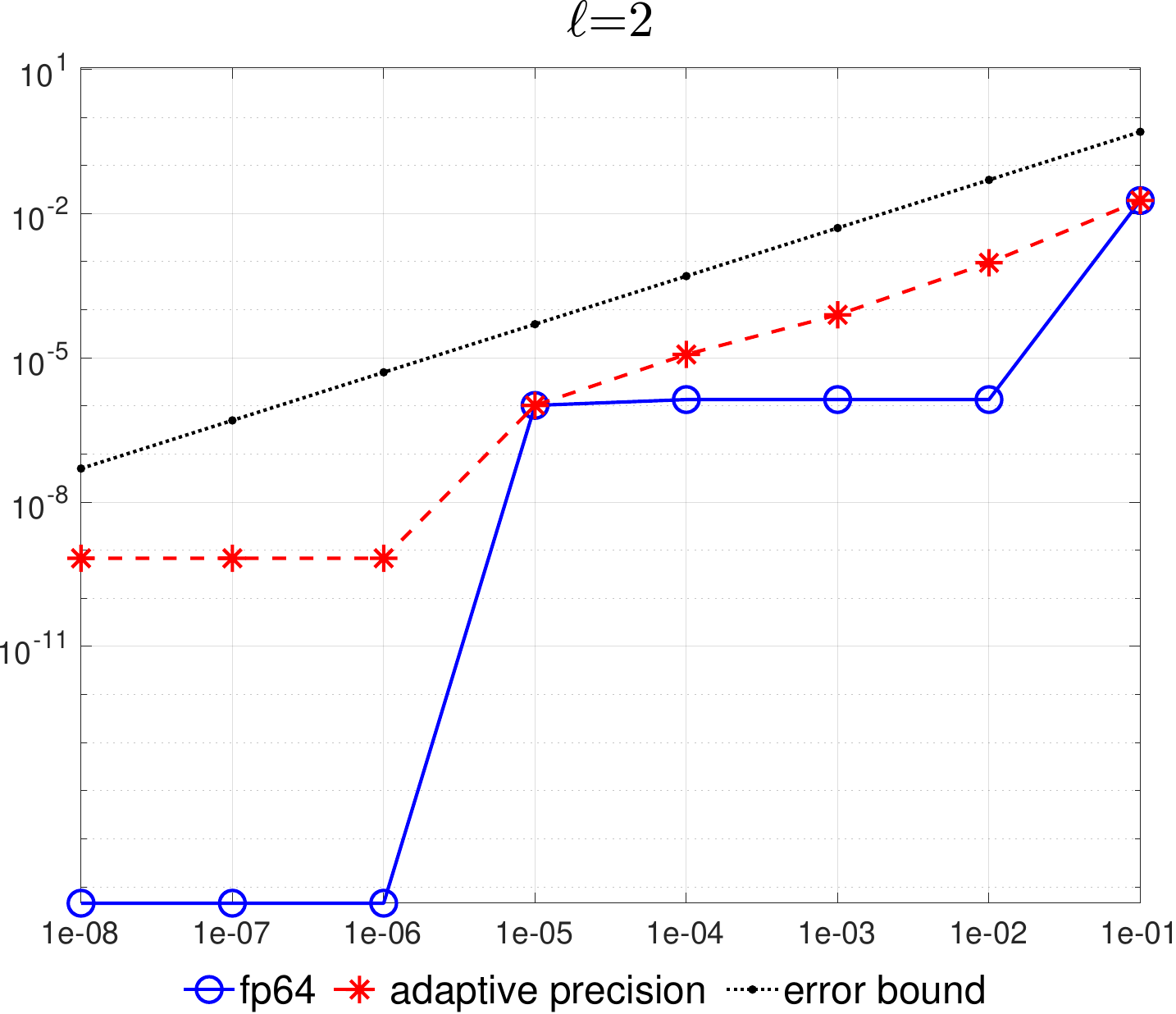}
\includegraphics[width=0.4\textwidth]{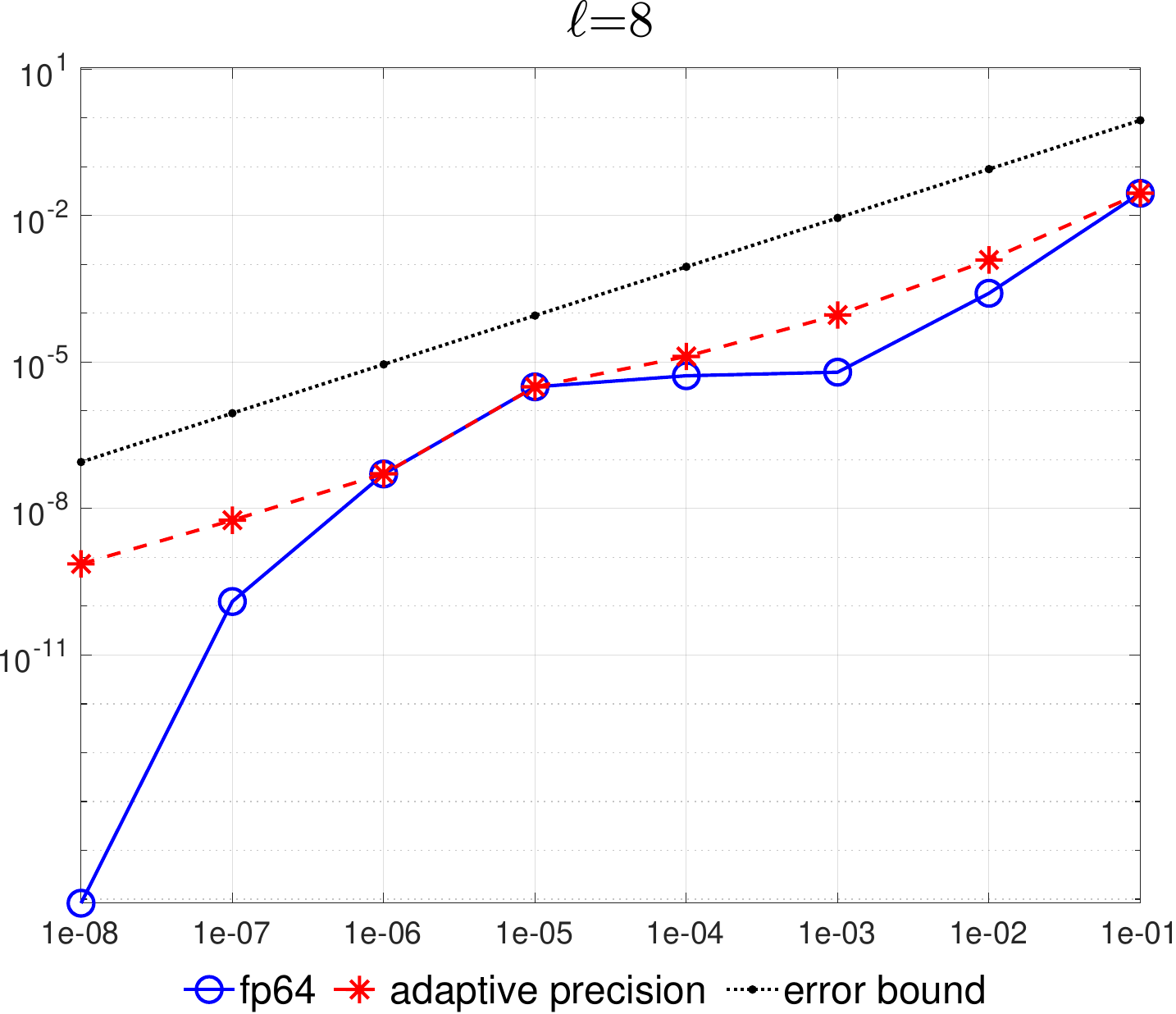}}
\caption{Reconstruction error for the adaptive-precision HODLR matrix. The $x$-axis indicates the value of $\varepsilon$ and the $y$-axis indicates relative global construction error.
}
\label{fig:gerr-lshp}
\end{figure}

We first examine the relative global construction error 
$\norm{H-\widehat{H}}_F/\norm{H}_F$
of $\widehat{H}$ produced by Algorithm~\ref{alg:mp-hodlr-adap} and the error bound~\eqref{eq:error-adap-prec-represent}; the global construction error for a uniform double precision scheme is also measured for comparison.
Fig.~\ref{fig:gerr-lshp} presents the results for HODLR matrices using two depths, $\ell=2$ and $\ell=8$. We note that we also ran the tests for other depth values and similar results were obtained.

It is interesting to note in~Fig.~\ref{fig:gerr-lshp} (a) and (c)
that for small $\varepsilon$, it is clear that the finite precision error dominates and there can be a huge difference between the mixed precision and the uniform double precision schemes. 
But as $\varepsilon$ grows larger, the low-rank approximation error starts to dominate in the bound, and the difference between the two schemes become marginal or unnoticeable. It is clear from the figure that for all values of $\varepsilon$, using double precision for storing the generators $U$ and $V$ for each layer takes more precision than necessary for satisfying the bound.

\subsection{Matrix--vector products}

\begin{figure}[t]
\centering
\subfigure[mat-1]{\includegraphics[width=0.38\textwidth]{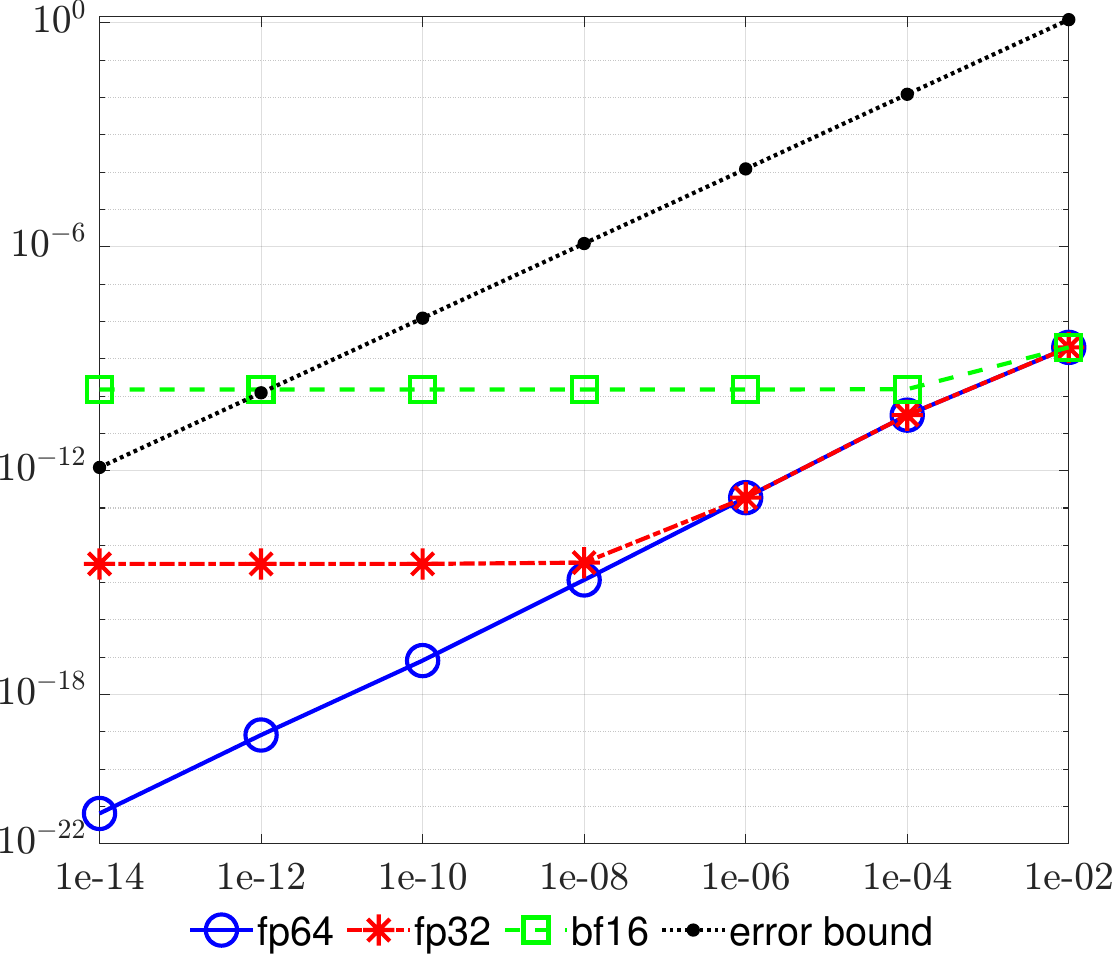}}
\subfigure[mat-2]{\includegraphics[width=0.38\textwidth]{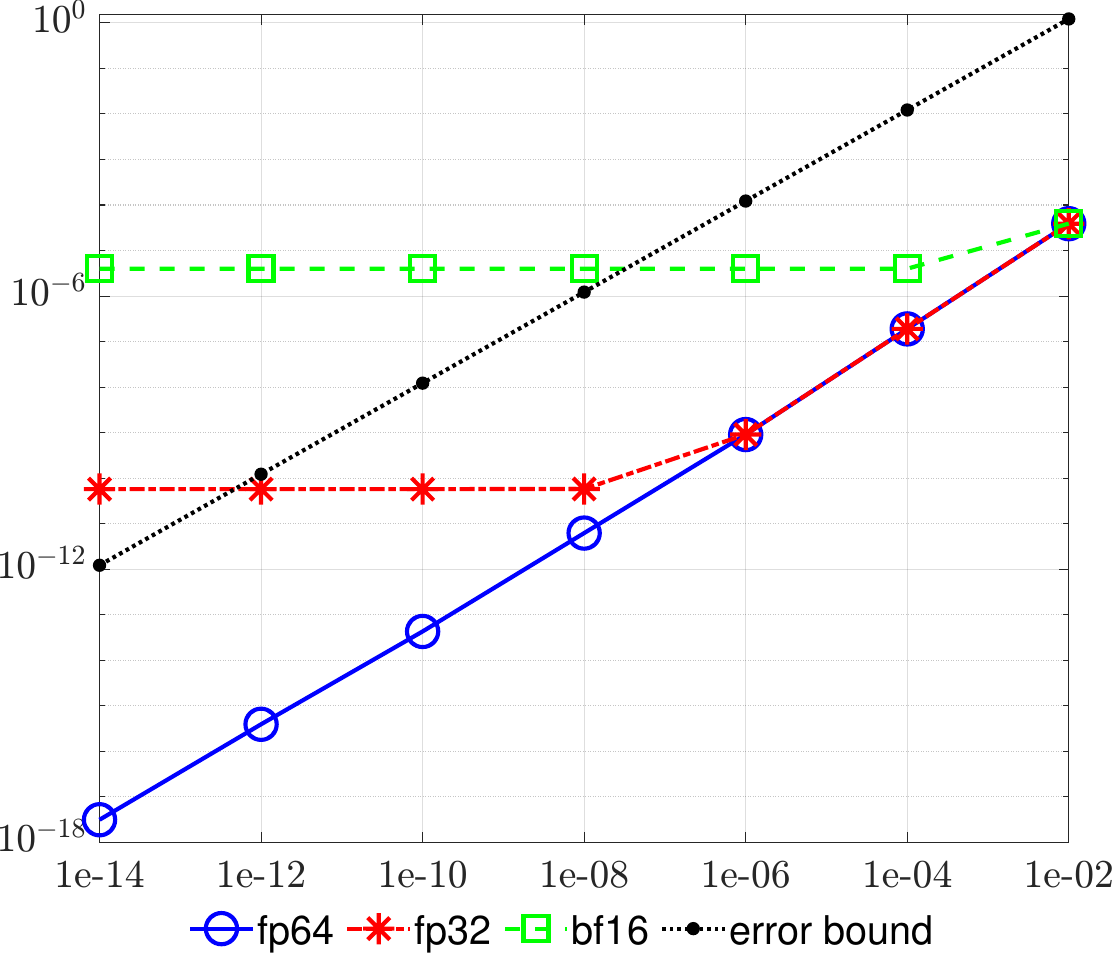}} \\
\vspace{-10pt}
\subfigure[mat-3]{\includegraphics[width=0.38\textwidth]{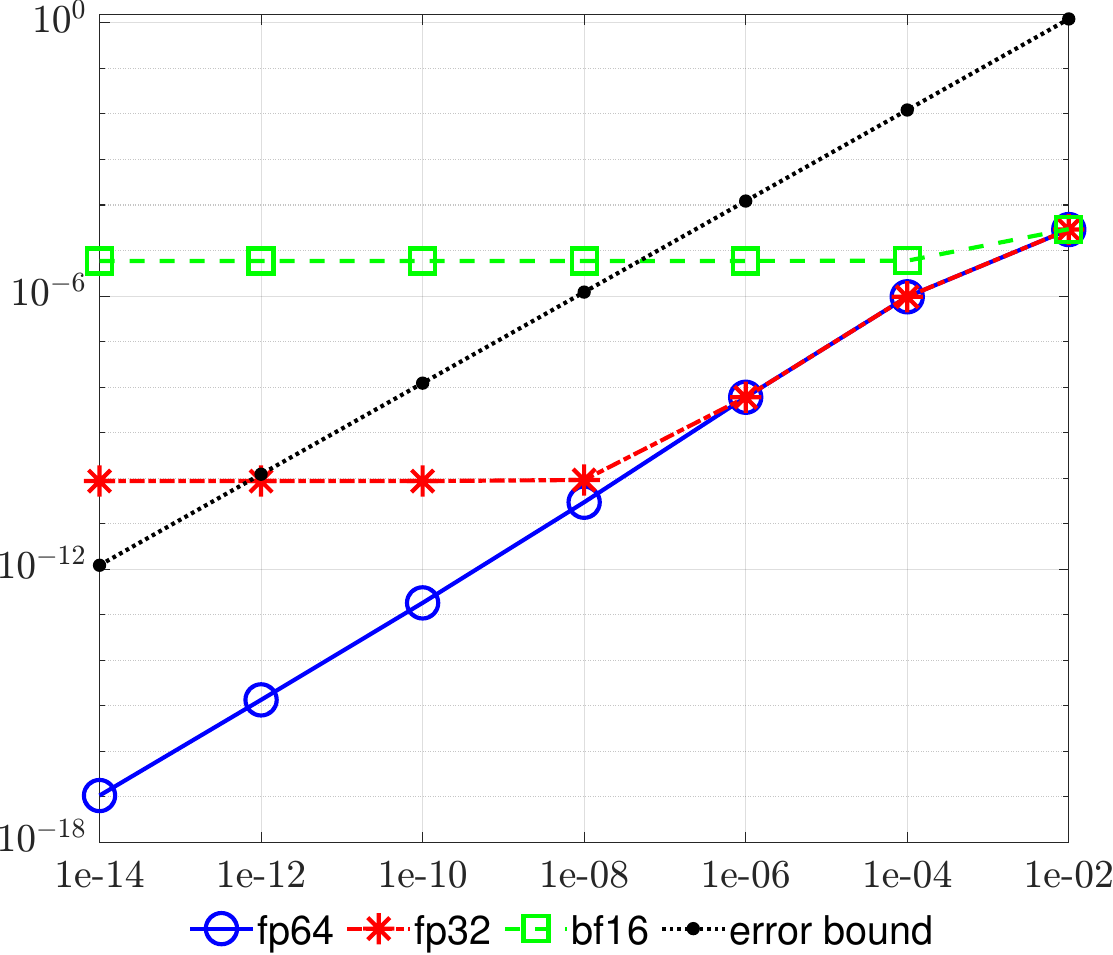}}
\subfigure[mat-4]{\includegraphics[width=0.38\textwidth]{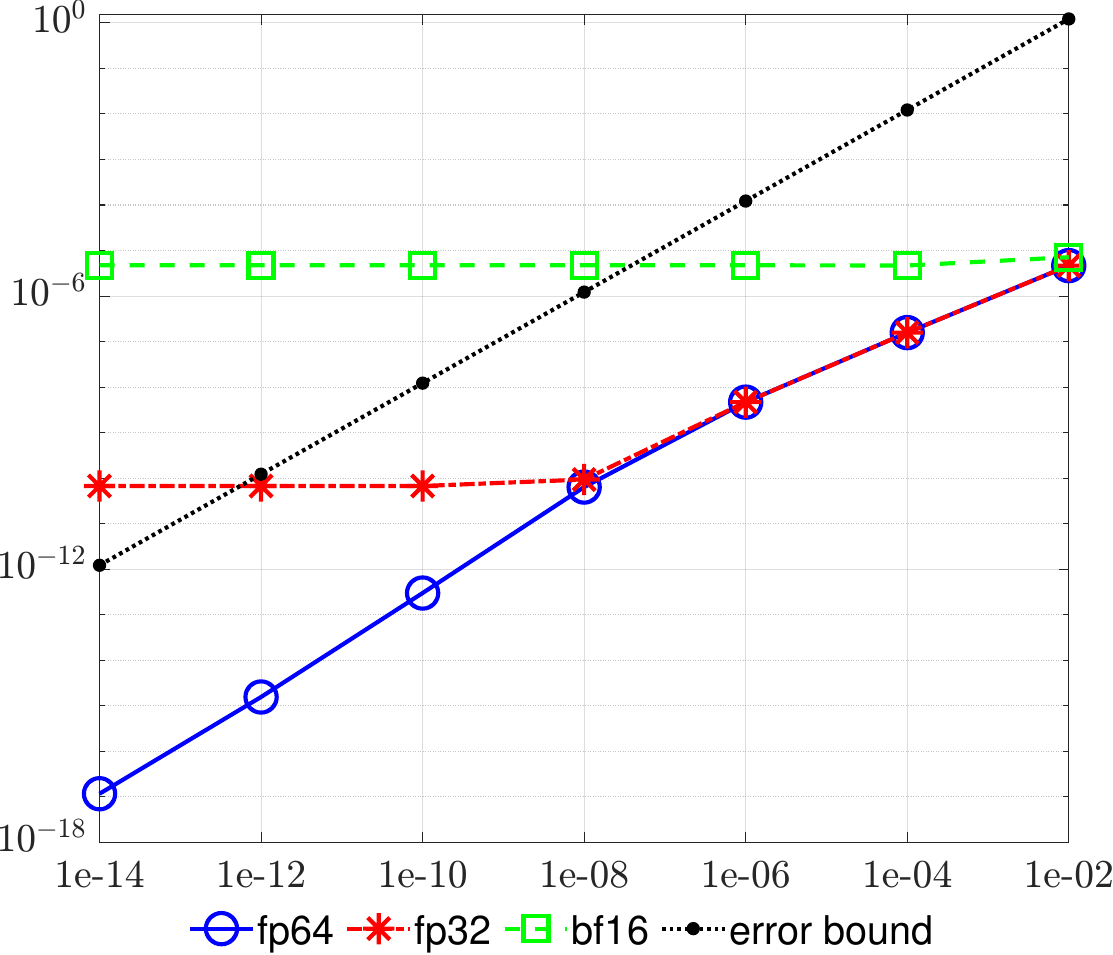}}
\caption{Backward error of matrix--vector products for mixed precision HODLR matrix with $\ell=8$; The $x$-axis indicates the value of $\varepsilon$ and the $y$-axis indicates the relative backward error.}
\label{fig:matvecerr}

\end{figure}

The pairwise interaction between objects such as $n$-body problems can be described by a kernel function. Associated with such a kernel function, a dense kernel matrix with dimension $n$ characterizes the interactions from a system of $n$ objects. We use the following kernels to demonstrate the results of Theorem~\ref{thm:mat-vec-bwerr}. 
\begin{equation}\label{eq:test-kernals}
\begin{aligned}
\text{(i)}\quad K_{ij} &= \begin{cases} 
      \frac{1}{x - y}, & \text{if} \quad x\neq y; \\
      1, & \text{otherwise}.
   \end{cases}
\qquad\text{(ii)}\quad K_{ij} = \begin{cases} 
      \log \|x_i - x_j\|_2, & \text{if} \quad x\neq y; \\
      0, & \text{otherwise}.
   \end{cases}\\
\text{(iii)}\quad K_{ij} &= \exp\Big( -\frac{\|x_i - x_j\|_2^2}{2h^2}  \Big).
\end{aligned}
\end{equation}

Our kernels are evaluated at 1D and 2D point sets: 
\begin{itemize}
    \item Set $\mathbf{s_1}$: A set of uniform grid points in $[0, 1]$.
    \item Set $\mathbf{s_2}$: A set of uniform grid points in $[-1, 1] \times [-1, 1]$.
\end{itemize}

We simulate tests on four kernel matrices of size $n=2000$, indicated by mat-1, mat-2, mat-3, and mat-4; mat-1 is generated by kernel (i) evaluated at point set $\mathbf{s}_1$; mat-2 is generated by kernel (ii) evaluated at point set $\mathbf{s}_2$; mat-3 and mat-4 are generated by kernel (iii) (defined by the Gaussian radial basis function) choosing $h=1$ and $h=20$, respectively, evaluated at point set $\mathbf{s}_2$.

We first form the matrix $\widehat{H}$ using the adaptive-precision approach in Algorithm \ref{alg:mp-hodlr-adap}. 
We then compute the matrix--vector product $\widehat{H}x=b$ where $\widehat{H}$ is
the computed approximation to the kernel matrix $K$, $x$ is the vector generated from the continuous uniform distribution in $(-1, 1)$, and $b$ is the result. We aim to demonstrate the balance between approximation error and finite precision error exhibited in Theorem \ref{thm:mat-vec-bwerr}, that is, we should choose the working precision such that $u\leq \varepsilon/n$ in order to not see the effects of finite precision.  We test a range of different values of $\varepsilon$ and three different working precisions $u$: fp64, fp32, and bf16. The relative backward error is evaluated as 
\begin{equation*}
    \frac{\|b - \widehat{H} x\|_F}{\|K\|_F \|x\|_F}.
\end{equation*}
All vectors $x\in \mathbb{R}^{n}$ are generated 10 times for each test and the backward errors are averaged. The results are shown in \figurename~\ref{fig:matvecerr}, where we only present the results for $\ell=8$; we also tested other
depth settings, e.g., $\ell=2$ and $\ell=5$, for the HODLR matrix construction and obtained similar results.

We can see from the plots that in order for the approximation error stemming from the choice of $\varepsilon$ to dominate, we do in fact need to have $u \lesssim \varepsilon$. For example, looking at the plot for mat-3, when we use single precision as the working precision, the error from finite precision dominates until around $\varepsilon = 10^{-8}$, corresponding to the unit roundoff for single precision. For working precision set to bf16, the finite precision rounding error dominates until $\varepsilon$ reaches the unit roundoff of bf16; see Table \ref{table:unitroundoff}, For double precision, the unit roundoff is small enough that the approximation error dominates the roundoff error for all tested values of $\varepsilon$. 

We also plot the bound from Theorem~\ref{thm:mat-vec-bwerr}; indeed, as given by the theoretical analysis, the bound holds as long as $u \leq\varepsilon/n$. In fact, due to the analysis, the bound is pessimistic and holds for $\varepsilon$ even smaller than $un$.
In general, it is clear that the smaller the value of  $\varepsilon$ for off-diagonal block truncation, the higher the precision we must use if we wish to not see the effects of finite precision error.  

\new{We have also evaluated the three kernel functions from 
(ii) and (iii) on the 3D point set $[-1, 1] \times [-1, 1] \times [-1, 1]$ and found similar behavior of the bound and the errors as in Fig.~\ref{fig:matvecerr}.}

\subsection{LU factorization}

\begin{figure}[t]
\centering
\subfigure[ex37]{
\includegraphics[width=0.4\textwidth]{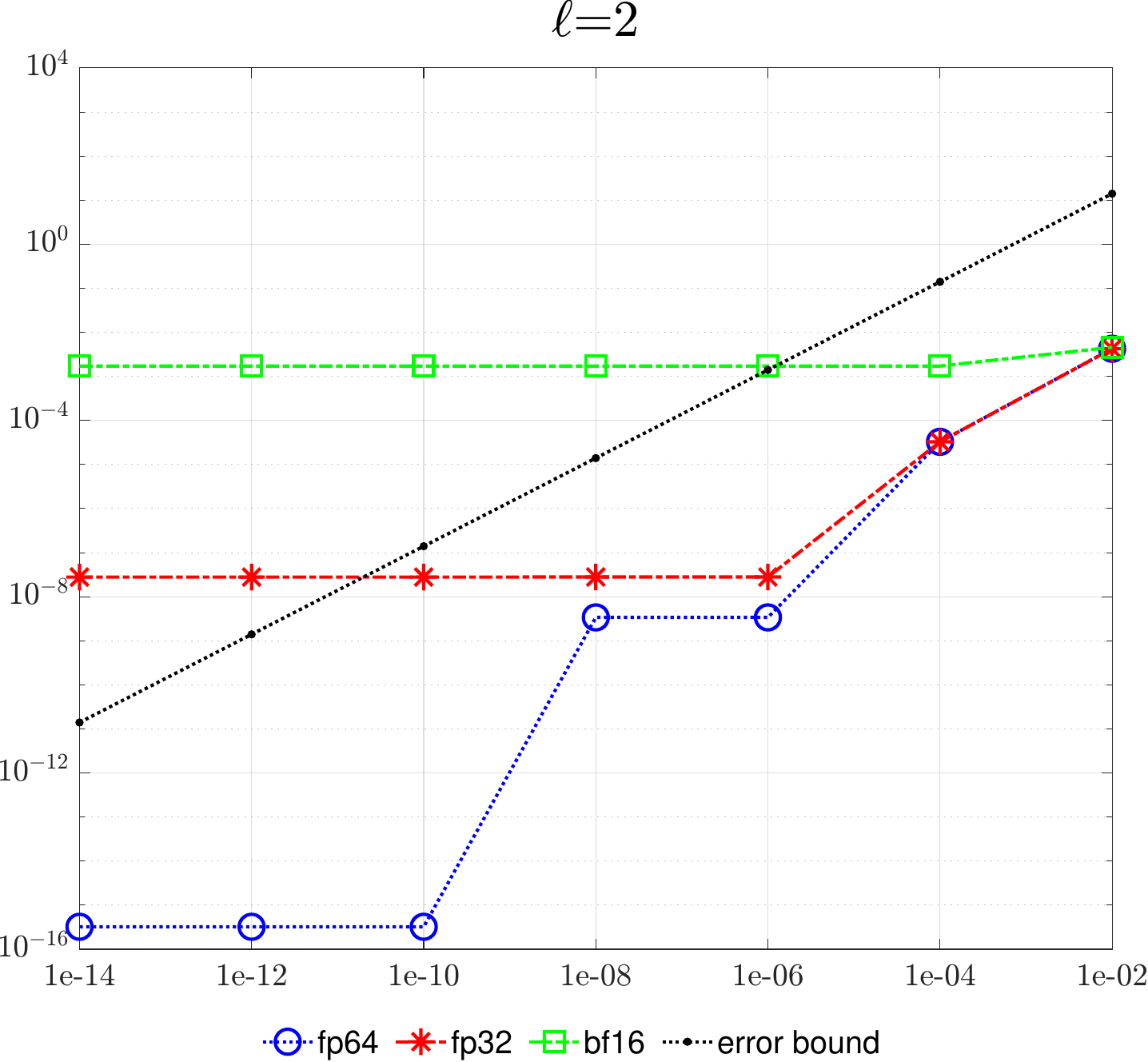}
\includegraphics[width=0.4\textwidth]{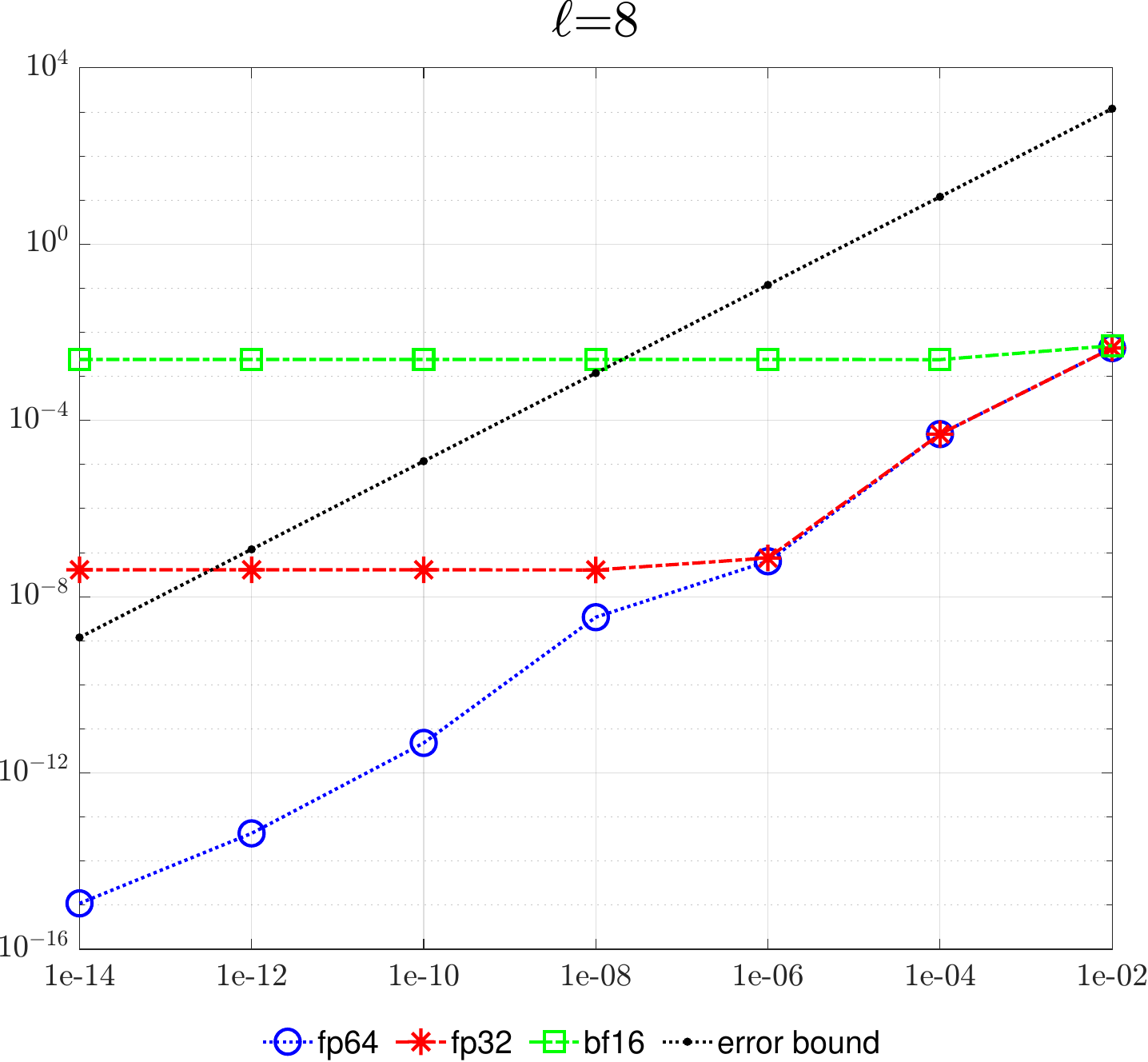}}
\\
\vspace{-10pt}
\subfigure[P64]{
\includegraphics[width=0.4\textwidth]{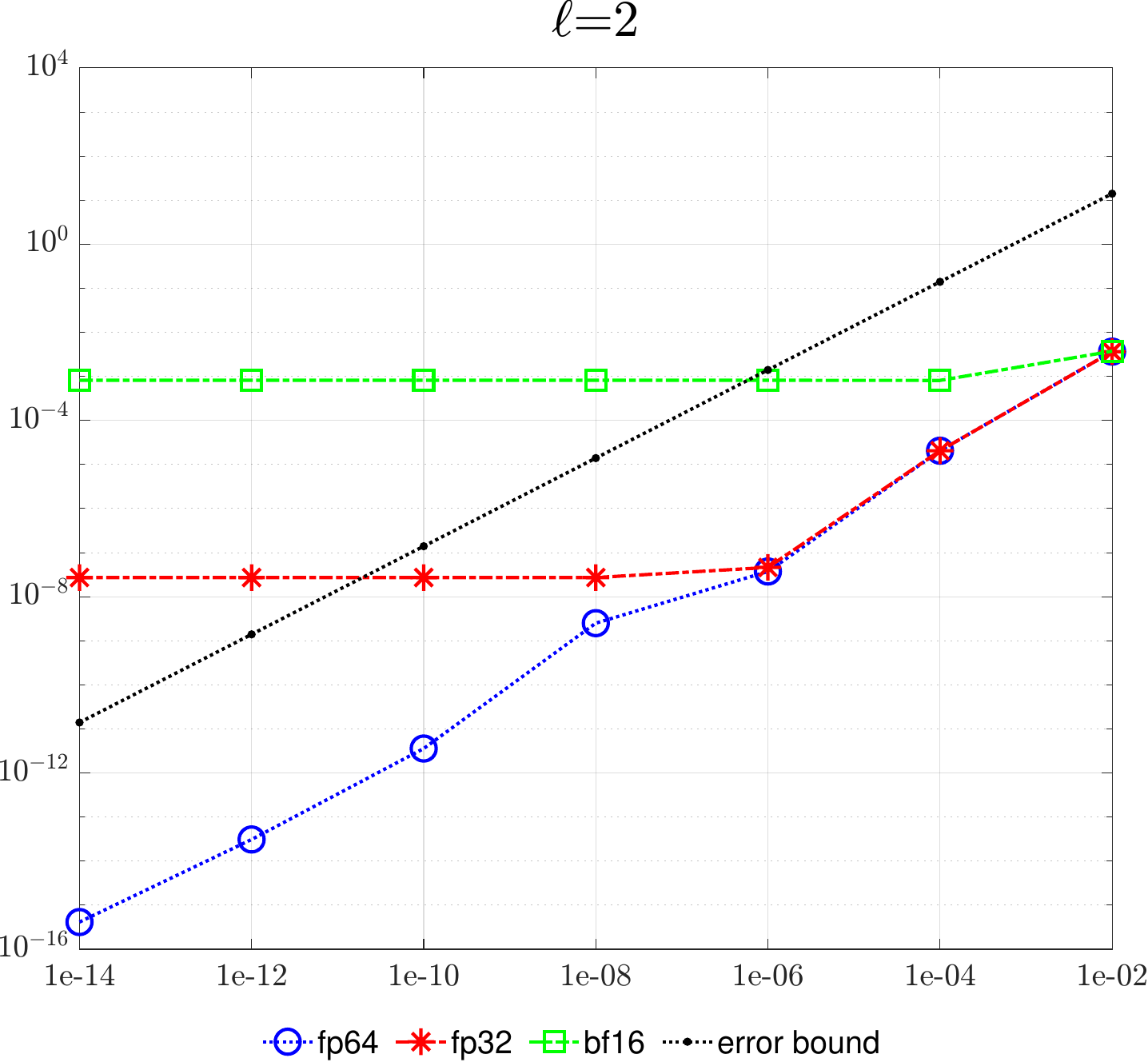}
\includegraphics[width=0.4\textwidth]{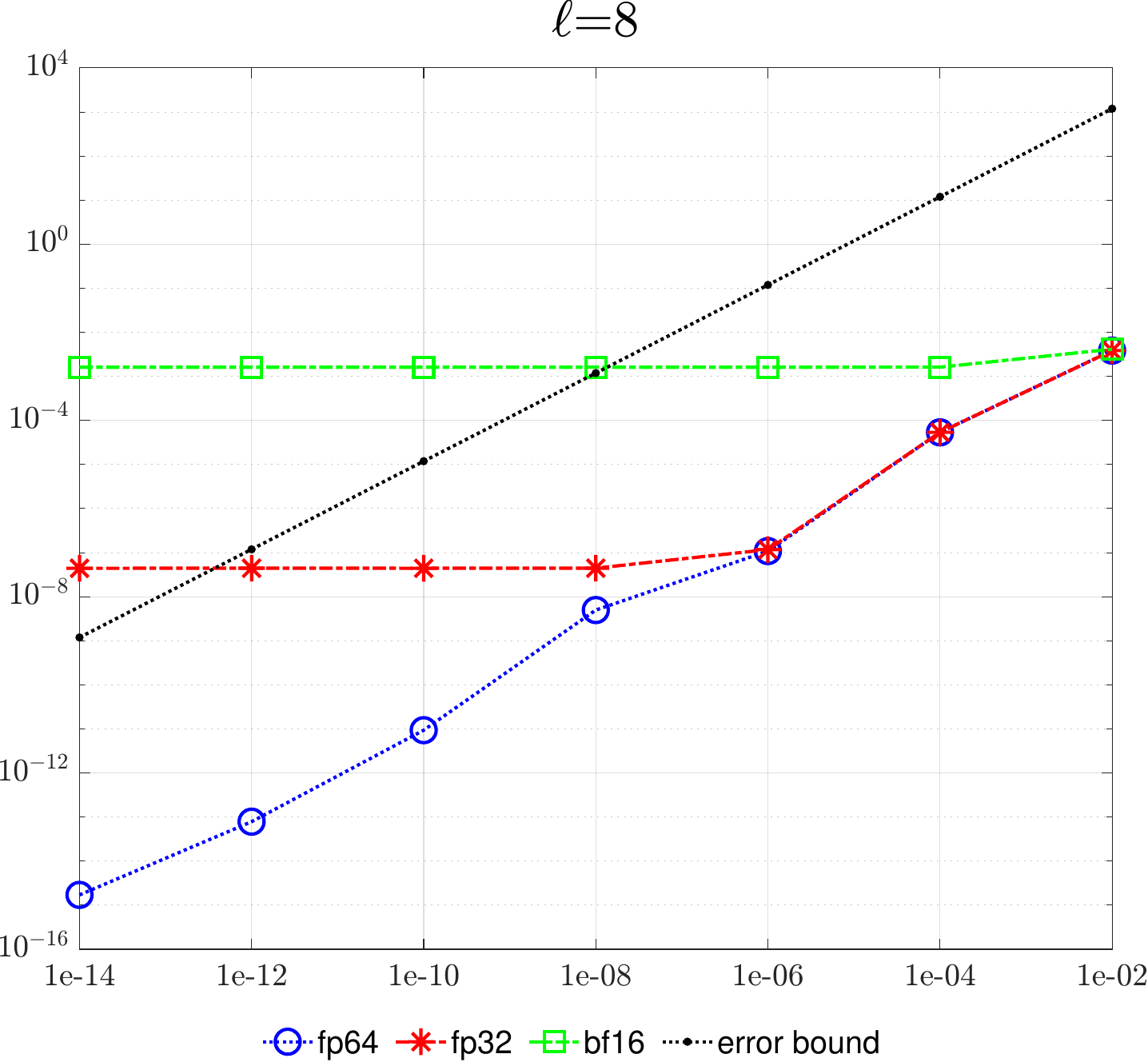}}
\\
\vspace{-10pt}
\subfigure[psmigr\_1]{
\includegraphics[width=0.4\textwidth]{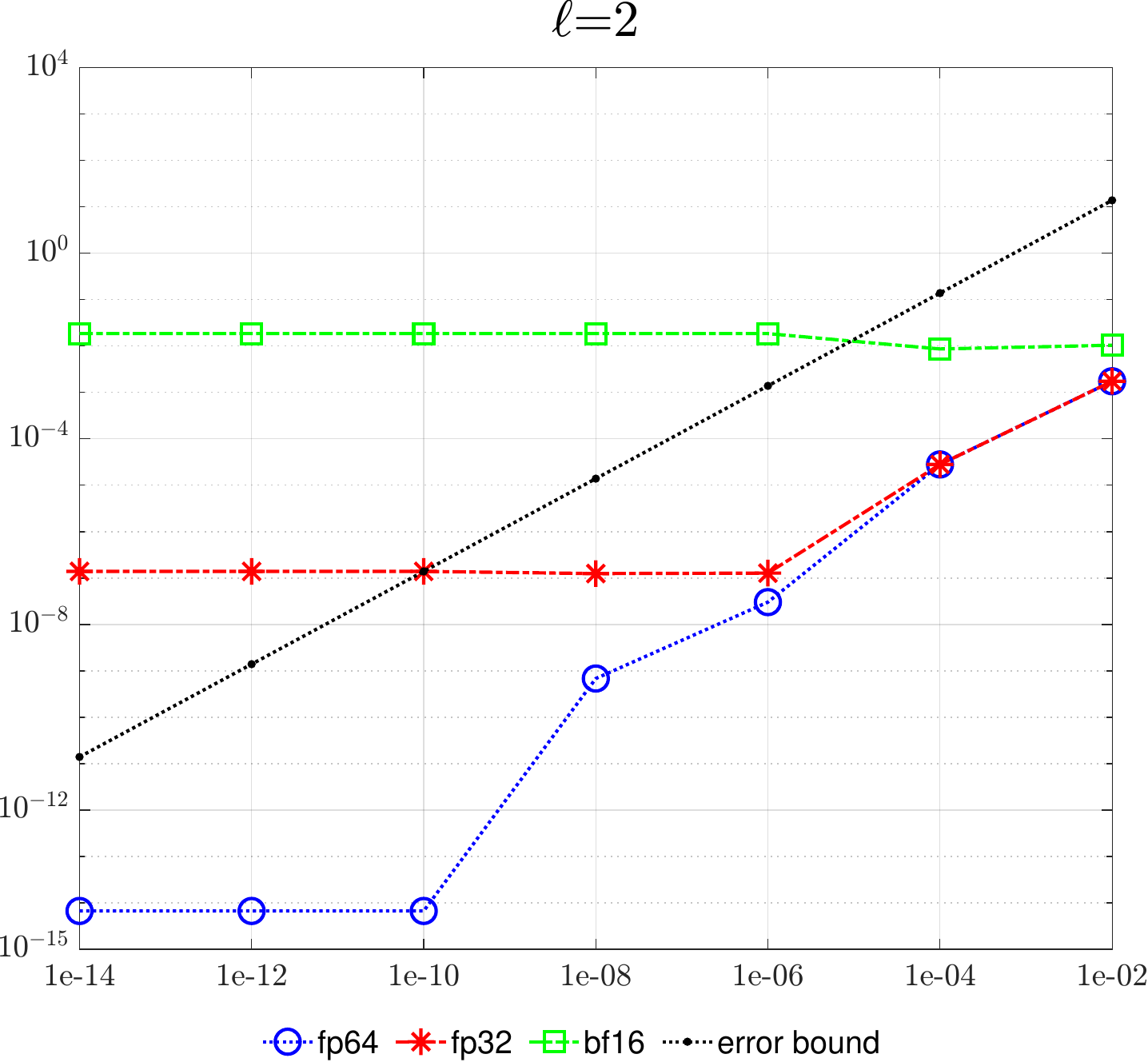}
\includegraphics[width=0.4\textwidth]{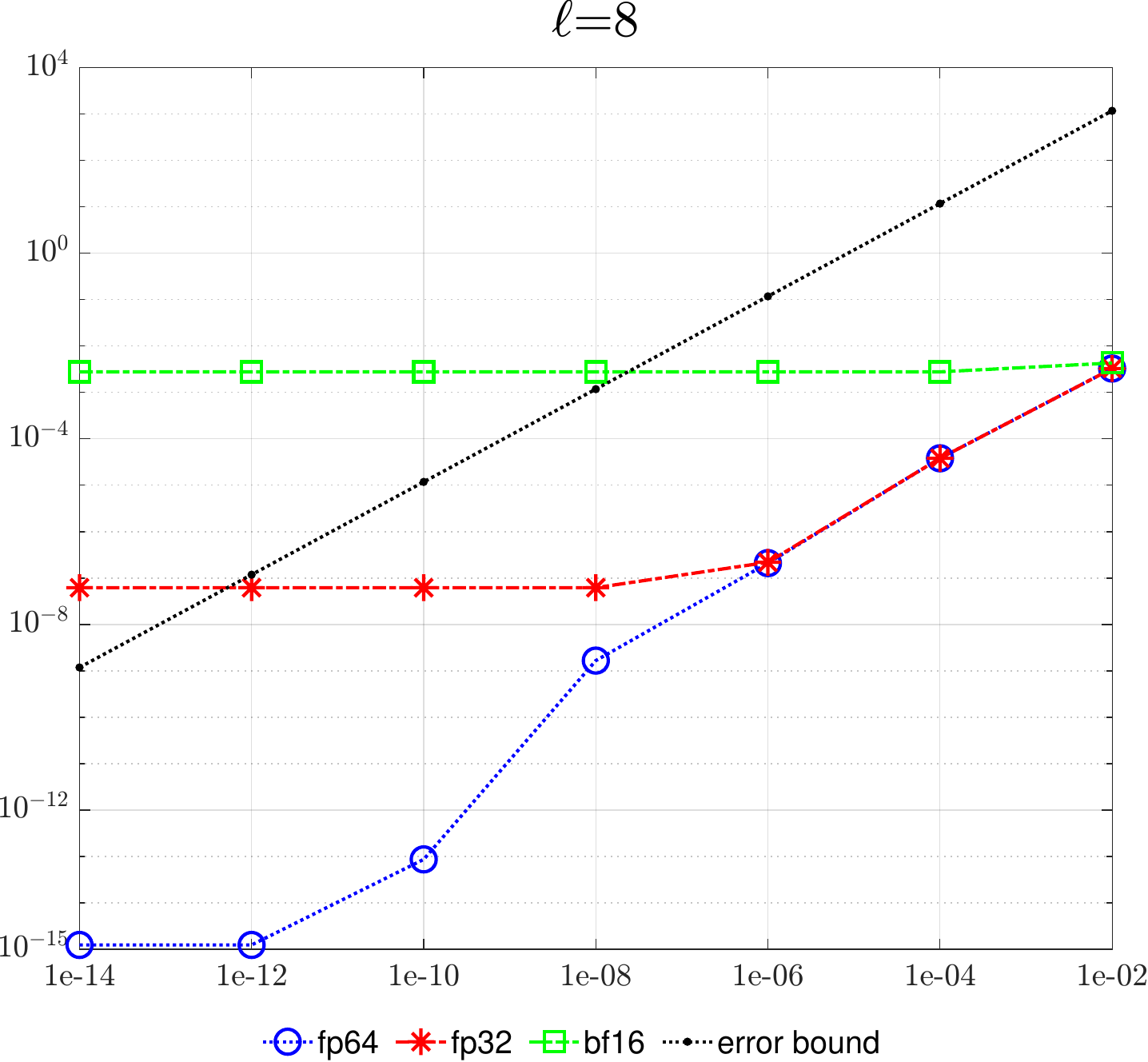}}
\caption{Backward error of LU factorization for mixed precision HODLR matrix. The $x$-axis indicates the value of $\varepsilon$ and the $y$-axis indicates relative backward error.
}
\label{fig:lu}
\end{figure}

We now evaluate the backward error of LU factorization, described in Theorem \ref{thm:luhodlr} for mixed-precision HODLR matrices, measured as
    $\|\widehat{L} \widehat{U} - A\|_F/\|A\|_F$, 
where $\widehat{L}$ and $\widehat{U}$ are the computed factors (corresponding to lower triangular and upper triangular parts) of hierarchical LU factorization of the computed HODLR matrix $\widehat{H}$ which approximates $A$.

As above, the mixed precision HODLR matrix is constructed using our adaptive-precision scheme (Algorithm~\ref{alg:mp-hodlr}), and the LU factorizations are separately evaluated using three different working precisions, double (fp64), single (fp32), and bfloat16 (bf16). We show results for two depths, 2 and 8.
The results are shown in \figurename~\ref{fig:lu}, where we plot the measured backward error ($y$-axis) for various values of $\varepsilon$ ($x$-axis). We also plot the line for \eqref{eq:lu_col} as error bound\footnote{Note that we only plot the line for the error bound of the LU factorization working in double precision since the two other lines are almost identical.} for $u \lesssim \varepsilon/n$.
We observe similar behavior as for the matrix--vector products. That is, the bounds are valid although pessimistic, and it is clear that the finite precision error dominates until around the level $u\approx \varepsilon$. Again, it is clear that a smaller value of $\varepsilon$ for off-diagonal block truncation requires a higher precision for computing LU factorization to ensure that the effects of finite precision do not dominate the backward error.

\subsection{Theoretical storage}

\begin{figure}[t]
\centering
\subfigure[bcsstk08]{\includegraphics[width=0.2\textwidth]{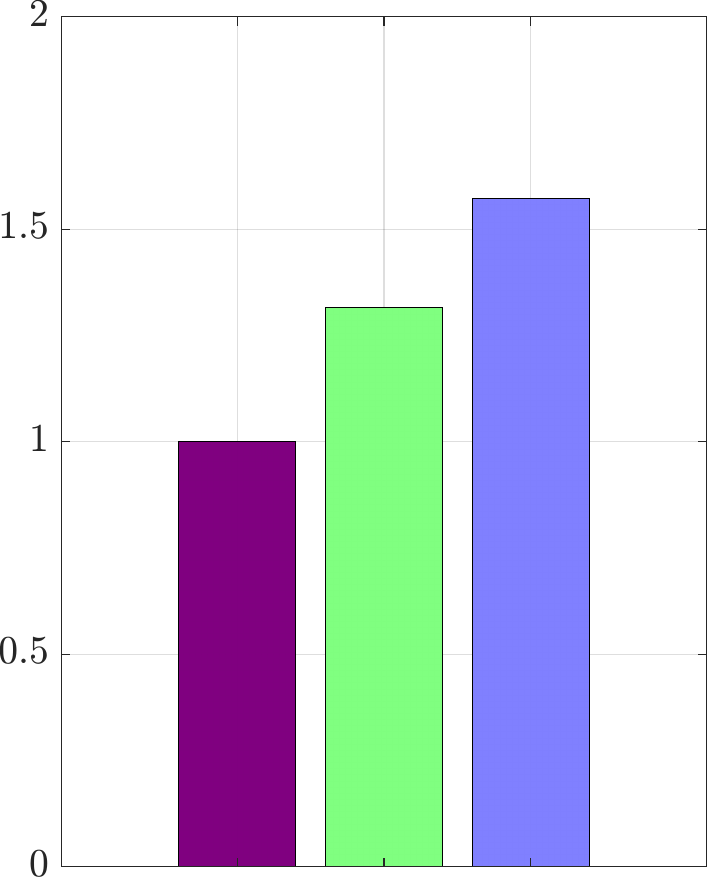}}
\subfigure[cavity18]{\includegraphics[width=0.2\textwidth]{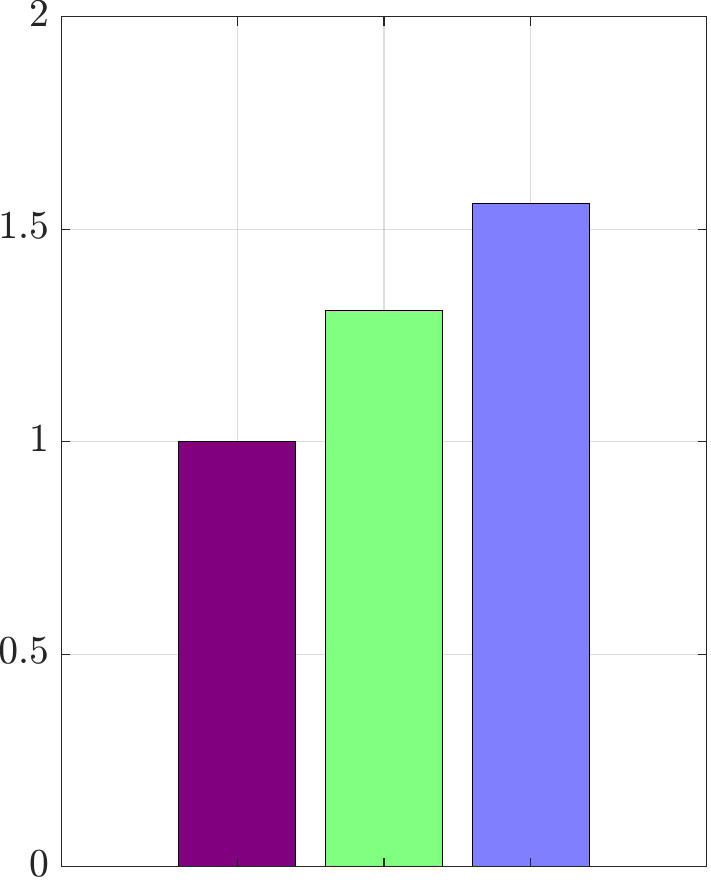}}
\subfigure[ex37]{\includegraphics[width=0.2\textwidth]{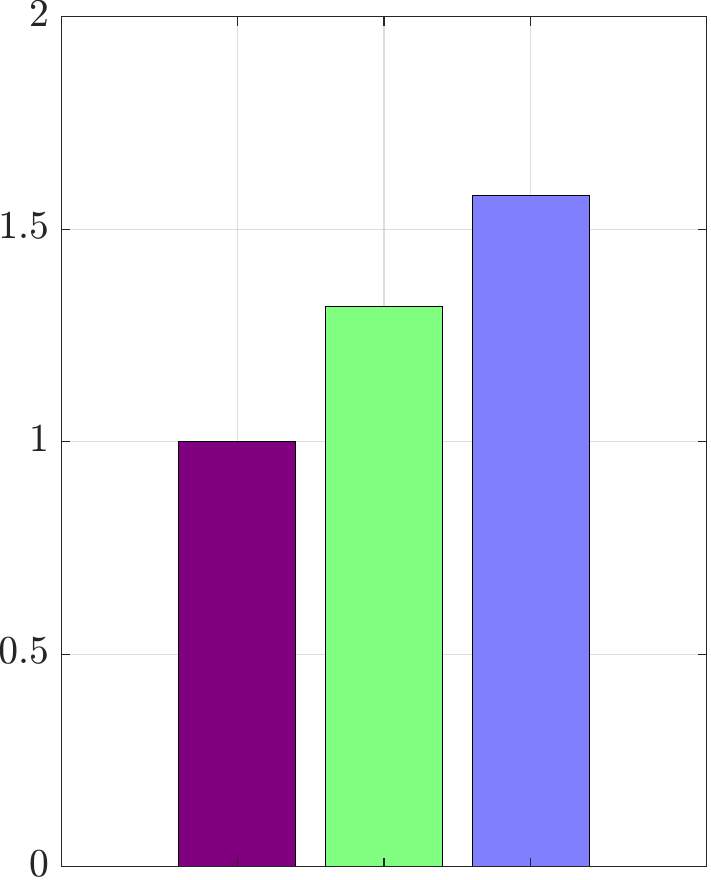}}
\subfigure[LeGresley\_2508]{\includegraphics[width=0.2\textwidth]{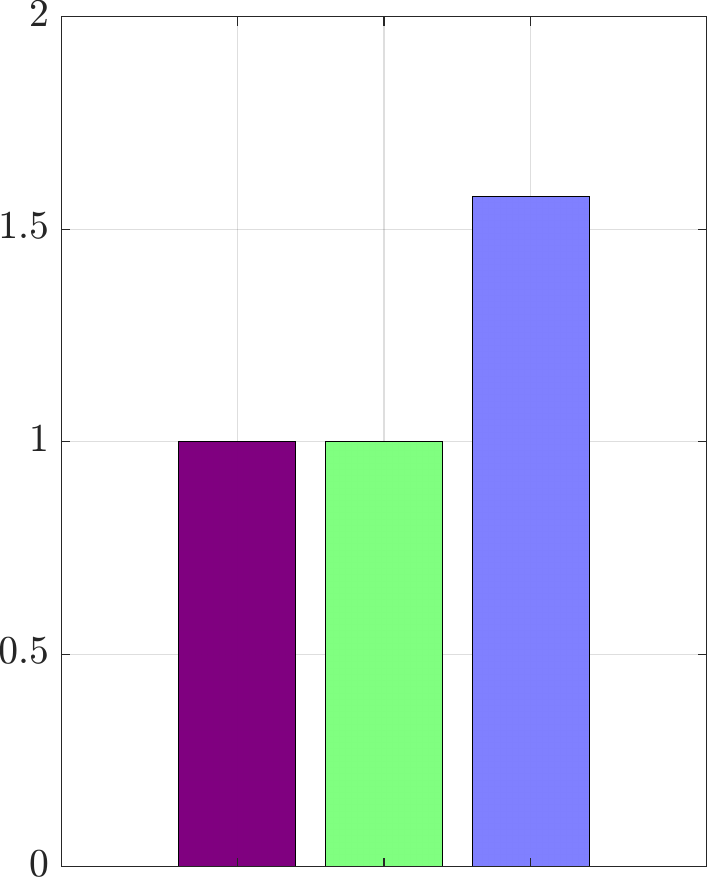}}
\\
\vspace{-8pt}
\subfigure[P64]{\includegraphics[width=0.2\textwidth]{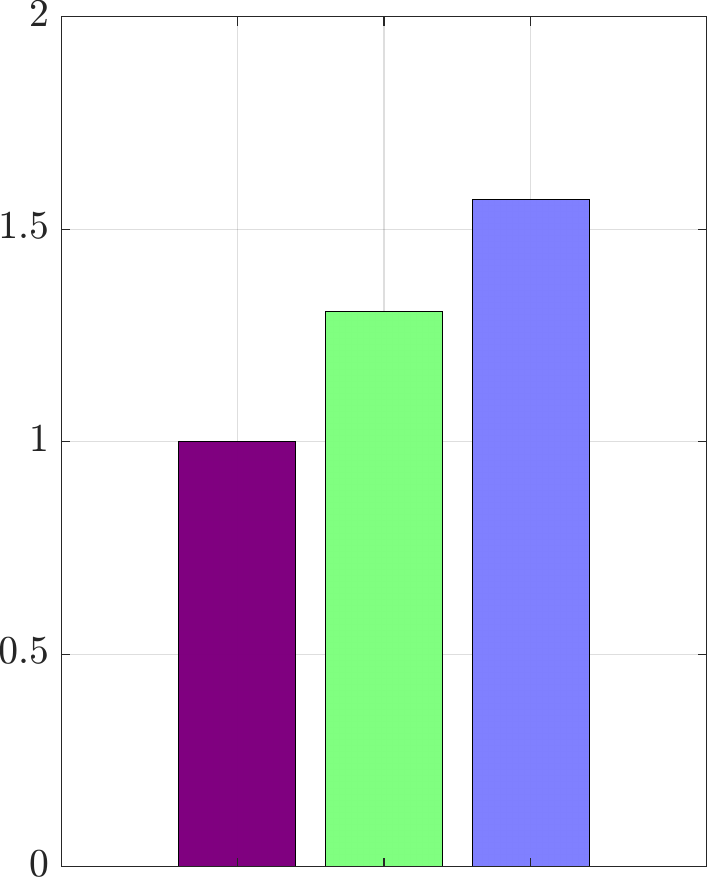}}
\subfigure[psmigr\_1]{\includegraphics[width=0.2\textwidth]{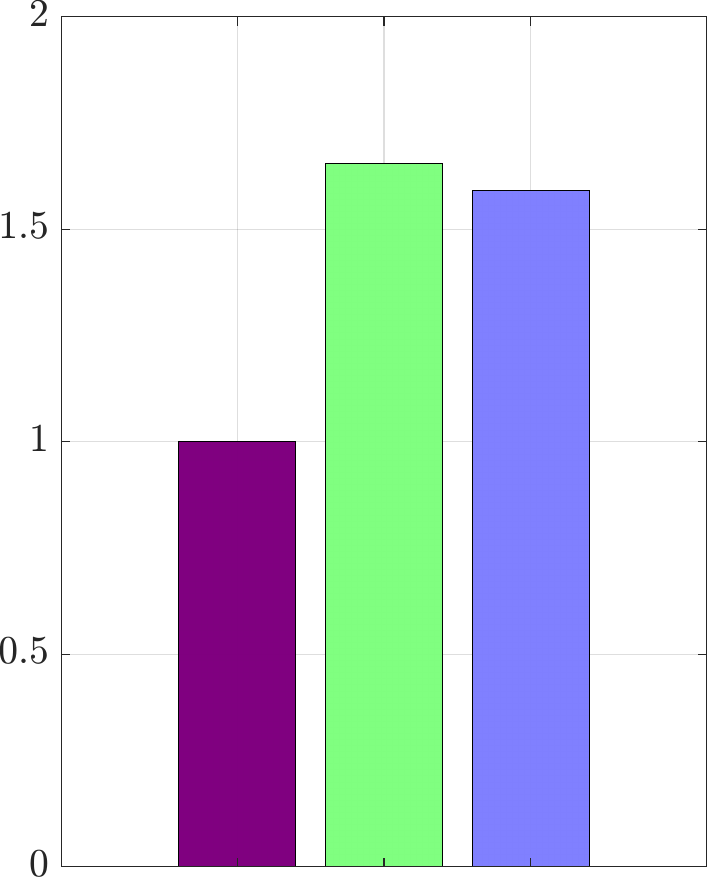}}
\subfigure[saylr3]{\includegraphics[width=0.2\textwidth]{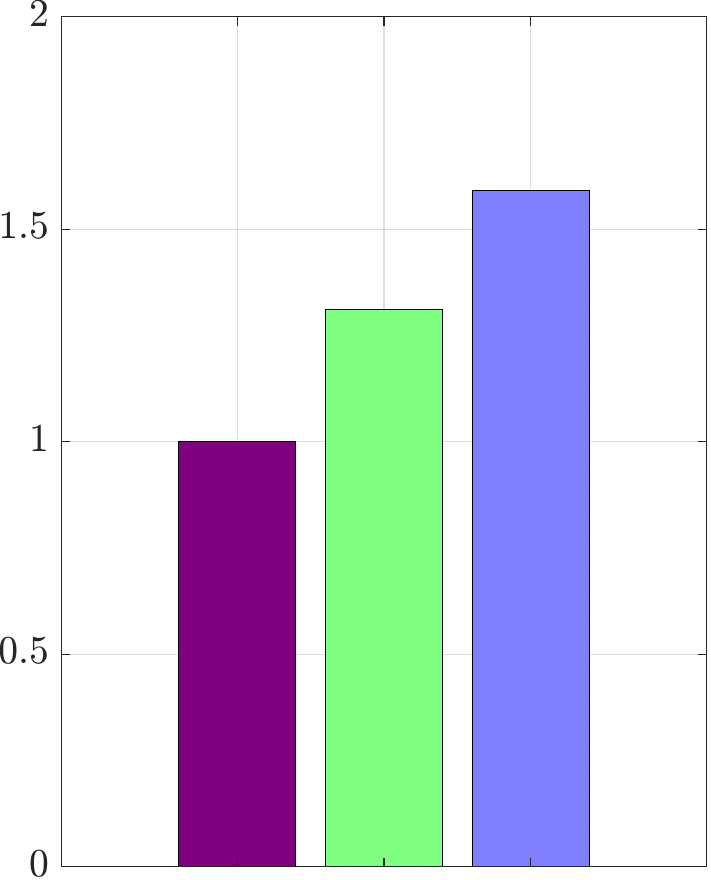}}
\subfigure[1138\_bus]{\includegraphics[width=0.2\textwidth]{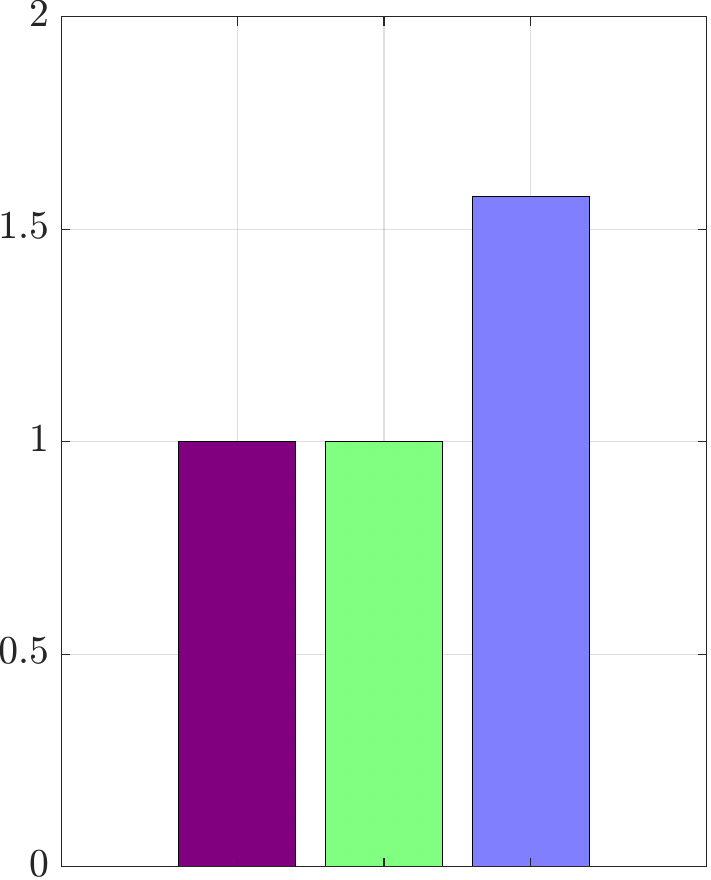}}

\caption{\new{Storage savings of adaptive-precision HODLR matrices relative to uniform (double) precision HODLR matrices}. The depth $\ell=8$; purples bars correspond to $\varepsilon=10^{-7}$, green bars indicate $\varepsilon=10^{-4}$, and blue bars indicate $\varepsilon=10^{-1}$.}
\label{fig:st-lshp}

\end{figure}

Our mixed-precision HODLR matrix use potentially lower precision to store the matrices $U$ and $V$, which can reduce the number of bits required for storing the HODLR matrix. We now evaluate how much we can reduce the storage cost by using our adaptive-precision construction approach (Algorithm~\ref{alg:mp-hodlr-adap}) compared to HODLR matrices stored entirely in double precision.  Our storage cost $\mathcal{S}$ is computed as total bits used for storing the matrices $U, V, D$. We compute the ratio of bits used in storing the double precision HODLR matrix to bits used in storing the adaptive-precision HODLR matrices (constructed using $\varepsilon=10^{-7}$, $10^{-4}$, and $10^{-1}$, corresponding to purple, green, and blue colors, respectively), for the eight test matrices listed in \tablename~\ref{tab:testmat}. 
\new{Note that both the uniform and mixed-precision HODLR matrices 
change with different choice of $\varepsilon$.}
The experimental results using HODLR matrices of depth $\ell=8$ are illustrated in \figurename~\ref{fig:st-lshp}; we also tested depths $\ell=2$ and $\ell=5$, which show only minor differences from $\ell=8$ and thus are not presented.

\new{Interestingly, for the Schur complement of the test matrix psmigr\_1,
the relative memory savings for $\varepsilon = 10^{-4}$ is slightly more than for $\varepsilon = 10^{-1}$.
A way to interpret this result is that, unlike in the other cases, the effect of storage reduction by using low precisions in the construction is lesser than the effect of storage reduction from coarser low-rank approximations, which are a consequence of larger $\varepsilon$.}
For smaller values of $\varepsilon$, i.e., $\varepsilon=10^{-7}$, the adaptive-precision scheme resulted in only double precision being used on our test set, and thus there are no storage savings. However, for larger values of $\varepsilon$, we can see that compared to storing the HODLR matrix in double precision, using our adaptive-precision HODLR matrix can effectively reduce the storage requirements, \new{typically by a factor of around $1.5$.}

\section{Conclusions}
\label{sec:conclude}
In this paper, we developed an adaptive-precision approach for constructing and storing HODLR matrices. We have analyzed the global representation error as well as the  backward error of HODLR matrix computations, including matrix--vector products and LU factorization using the mixed precision representation. Our analysis essentially shows that as long $u\lesssim \varepsilon/n$, where $u$ is the working precision and $\varepsilon$ is the approximation parameter for the HODLR matrix, the backward errors in these computations depend only on $\varepsilon$; in other words, the larger the approximation parameter, the lower the precision that one can use without affecting the backward error. These observations are confirmed by our numerical experiments. Our evaluation of storage cost further shows that our adaptive-precision scheme can reduce the storage cost significantly. We note that all the results in this paper are also valid when a uniform precision is used for storing and computing with HODLR matrices, and thus we also provide new backward error results for this case. 

The fused multiply-add plays a critical role in the matrix computations described in our paper, e.g., the Schur complement in the LU factorization; future work includes a rounding error analysis of the fused multiply-add operation based on our mixed precision HODLR matrix format.

\appendix

\bibliographystyle{siamplain}
\bibliography{references}
\end{document}


\maketitle

\section{A detailed example}

Here we include some equations and theorem-like environments to show
how these are labeled in a supplement and can be referenced from the
main text.
Consider the following equation:
\begin{equation}
  \label{eq:suppa}
  a^2 + b^2 = c^2.
\end{equation}
You can also reference equations such as \cref{eq:matrices,eq:bb} 
from the main article in this supplement.

\lipsum[100-101]

\begin{theorem}
  An example theorem.
\end{theorem}

\lipsum[102]
 
\begin{lemma}
  An example lemma.
\end{lemma}

\lipsum[103-105]

Here is an example citation: \cite{KoMa14}.

\section[Proof of Thm]{Proof of \cref{thm:bigthm}}
\label{sec:proof}
\lipsum[106-112]

\section{Additional experimental results}
\Cref{tab:foo} shows additional
supporting evidence. 

\begin{table}[htbp]
{\footnotesize
  \caption{Example table}  \label{tab:foo}
\begin{center}
  \begin{tabular}{|c|c|c|} \hline
   Species & \bf Mean & \bf Std.~Dev. \\ \hline
    1 & 3.4 & 1.2 \\
    2 & 5.4 & 0.6 \\ \hline
  \end{tabular}
\end{center}
}
\end{table}

\bibliographystyle{siamplain}
\bibliography{references}